\documentclass[11pt,letterpaper]{article}
\pdfoutput=1 
\usepackage[margin=3.45cm]{geometry}
\usepackage{amsmath, amssymb, amsthm}
\usepackage{graphicx, color}
\usepackage{array, multirow}
\usepackage[font=small,labelfont=bf]{caption} 
\usepackage{lipsum}
\usepackage{amsfonts}
\usepackage{tabularx}
\usepackage{booktabs}
\usepackage{graphicx}
\usepackage{epstopdf}
\usepackage{algorithmic}
\usepackage{tikz}
\ifpdf
  \DeclareGraphicsExtensions{.eps,.pdf,.png,.jpg}
\else
  \DeclareGraphicsExtensions{.eps}
\fi

\usepackage[shortlabels]{enumitem}

\usepackage[sort]{natbib}

\setcitestyle{numbers,square,comma}
\usepackage{hyperref}
\hypersetup{colorlinks=true,
			urlcolor=blue,
			linkcolor=blue,
			citecolor=blue,
			bookmarksdepth=paragraph}
\usepackage[nameinlink]{cleveref}
\crefname{equation}{}{}
\crefname{section}{section}{sections}
\crefname{figure}{figure}{figures}
\crefname{table}{table}{tables}
\crefname{example}{example}{examples}
\crefname{proposition}{proposition}{propositions}
\Crefname{assumption}{assumption}{assumptions}
\Crefname{section}{Section}{Sections}
\Crefname{figure}{Figure}{Figures}
\Crefname{table}{Table}{Tables}
\Crefname{definition}{Definition}{Definitions}
\Crefname{theorem}{Theorem}{Theorems}
\Crefname{remark}{Remark}{Remarks}
\Crefname{example}{Example}{Examples}
\Crefname{proposition}{Proposition}{Propositions}
\Crefname{assumption}{Assumption}{Assumptions}
\numberwithin{equation}{section}

\newtheorem{theorem}{Theorem}[section]

\newtheorem{proposition}{Proposition}[section]
\theoremstyle{definition}

\newtheorem{remark}{Remark}[section]
\newtheorem{assumption}{Assumption}[section]

\title{Data-driven Discovery of Invariant Measures}
\author{Jason J.\ Bramburger$^1$, Giovanni Fantuzzi$^2$}
\date{
	$^1$Department of Mathematics and Statistics, Concordia University,\\ Montr\'eal, QC, Canada
	\\
	$^2$Department of Mathematics, Friedrich--Alexander--Universit\"at Erlangen--N\"urnberg, Germany
}

\begin{document}

\maketitle

\begin{abstract}
	\noindent
Invariant measures encode the long-time behaviour of a dynamical system. In this work, we propose an optimization-based method to discover invariant measures directly from data gathered from a system. Our method does not require an explicit model for the dynamics and allows one to target specific invariant measures, such as physical and ergodic measures. Moreover, it applies to both deterministic and stochastic dynamics in either continuous or discrete time. We provide convergence results and illustrate the performance of our method on data from the logistic map and a stochastic double-well system, for which invariant measures can be found by other means. We then use our method to approximate the physical measure of the chaotic attractor of the R\"ossler system, and we extract unstable periodic orbits embedded in this attractor by identifying discrete-time periodic points of a suitably defined Poincar\'e map. This final example is truly data-driven and shows that our method can significantly outperform previous approaches based on model identification.
\end{abstract}

\section{Introduction}
Dynamical system analysis traditionally takes one of two complementary approaches.
The first is to study individual system trajectories in order to identify fundamental structures such as steady-states, periodic orbits, heteroclinic connections, and chaotic attractors. The second approach is to understand the statistical behaviour of the system by examining the evolution of measures under the dynamics. Paramount in the study of measure-theoretic properties of dynamical systems is the identification of \emph{invariant} measures, which are preserved by the flow of the system. Invariant measures lie at the centre of ergodic theory and provide information regarding the long-time behaviour of a system.
Indeed, ergodic theorems allow one to calculate infinite-time averages as expectations with respect to invariant measures~\cite{brin2002introduction}, transferring one from the realm of dynamics to one of probability theory.

In this paper, we describe an optimization-based approach to extract invariant measures from dynamic data. Such a problem goes back at least to the work of Ulam \cite{ulam1960collection}, who proposed discretizing the phase space into a finite collection of disjoint sets and using the data to approximate transition probabilities between these sets \cite{goswami2018constrained,eDMDconv2,klus2016towards,dellnitz1999approximation}. This method produces a square stochastic matrix that approximates the Perron--Frobenius operator, that is, the linear operator describing how probability measures evolve under the system dynamics (see \cref{PerronFrobenius} below for a precise definition). By the Perron--Frobenius theorem, this matrix has a positive eigenvector with eigenvalue 1 that approximates the density of an invariant measure of the system. However, the computational complexity of Ulam's original method grows cubically with the dimension of the phase space and, even though recent variations of the method exhibit quadratic scaling \cite{gerber2018scalable}, it quickly becomes prohibitive. Moreover, accurate approximations of invariant measures require fine discretizations of the phase space and a considerable amount of data.   

An alternative approach to approximating invariant measures from data rests on the fact that the Perron--Frobenius operator is the formal adjoint of the Koopman operator. The Koopman operator is linear and acts on real-valued functions, often called \emph{observables}, through composition with the flow of a dynamical system. Its action can be approximated on finite collections of observables, termed \emph{dictionaries}, using extended dynamic mode decomposition (EDMD) \cite{eDMD}. This method produces a square matrix that approximates the Koopman operator and any left eigenvector with eigenvalue close to 1 approximates the density of an invariant measure of the underlying dynamical system \cite{eDMDconv2,klus2016towards,goswami2018constrained}. Despite EDMD's rigorous convergence guarantees \cite{bramburger2023auxiliary,korda2018convergence,eDMD}, however, the existence of nonnegative eigenvectors with unit eigenvalues cannot be guaranteed. Thus, contrary to Ulam's method, one must either settle for signed densities coming from eigenvalues near 1, or augment EDMD to produce a column-wise stochastic approximate Koopman operator matrix \cite{gerber2018scalable}.

In this work, we propose to approximate invariant measures by combining EDMD with convex optimization. Precisely, we first employ a variation of the EDMD method from~\cite{bramburger2023auxiliary} to produce a rectangular (rather than square) matrix approximation for the action of the Koopman operator on a dictionary of polynomial observables. Using this approximation, we then provide a data-driven implementation of the convex optimization framework for studying invariant measures developed in \cite{Korda2021invariant} for dynamical systems with known polynomial governing equations (see also \cite{bose2007duality,bose2005minimum} for a different model-based implementation of the same framework, which allows for non-polynomial governing equations at the expense of other assumptions on the system's invariant measures). This strategy has two major advantages compared both to Ulam's method and to alternative EDMD-based techniques. First, we can easily target different invariant measures if more than one exists for the dynamical systems generating the data. In particular, we can search for \emph{ergodic} invariant measures, which are extreme points of the set of all invariant measures in the sense of convex analysis. Second, we are not limited to discovering densities of absolutely continuous invariant measures. In particular, we are able to discover singular measures supported  on invariant trajectories for the underlying dynamical system. 

The ability of our approach to discover extreme singular invariant measures from data opens up new possibilities for the analysis of dynamical systems. For example, extreme singular invariant measures for chaotic discrete dynamical systems are often supported on fixed points or on unstable periodic orbits (UPOs) that densely fill the system's attractor. Thus, our approach provides a way of systematically extracting UPOs from chaotic maps. We demonstrate this by extracting both short- and long-period UPOs for the continuous-time chaotic R\"ossler system using only data from a one-dimensional Poincar\'e section. Precisely, we first use our data-driven technique to extract extreme atomic measures supported on UPOs of the Poincar\'e return map of period up to 19. Then, we recover the corresponding continuous-time UPOs for the R\"ossler system. This strategy outperforms data-driven methods that first identify a model for the Poincar\'e map data \cite{bramburger2020poincare,bramburger2021data,bramburger2021deep}, because small inaccuracies in the identified model propagate to forward iterations and make it difficult to discover long-period UPOs. There are also advantages compared to related methods for extracting UPOs from approximate invariant measures in continuous-time \cite{lakshmi2020finding,lakshmi2021finding}, as these rely on the target UPO being well approximated by the zero level set of a nonnegative polynomial. This is usually not the case unless the polynomial degree is large, because continuous-time UPOs are not algebraic curves in general. In contrast, the corresponding UPO for the Poincar\'e map generates  an atomic measure, whose atoms can be represented exactly as the zeros of a nonnegative polynomial. The Poincar\'e map formulation is therefore better suited to algebraic methods and our data-driven approach allows one to implement it even when, as is typical, the Poincar\'e map is not known.

To make this article self-contained, \cref{sec:InvMeas} introduces invariant measures and reviews in detail the optimization framework for their approximation developed in \cite{Korda2021invariant}. \Cref{sec:InvMeasData} shows how to implement this framework in a data-driven setting by using EDMD, while convergence results for dynamical systems governed by unknown polynomial maps are proved in \cref{sec:Convergence}. In these sections we focus on discrete-time deterministic dynamics for simplicity, but we comment on extensions to continuous-time and stochastic dynamics in \cref{sec:Extensions}. \Cref{s:examples} demonstrates the power of our data-driven approach on a series of examples. Final remarks and open questions are offered in \cref{s:conclusion}.

\section{Invariant Measures}\label{sec:InvMeas}

We begin by reviewing invariant measures for discrete-time dynamical systems, as well as existing model-based optimization tools for their approximation. As mentioned above, these tools are well known (see, e.g., \cite{Henrion2012semidefinite,Chernyshenko2014rsta,Fantuzzi2016siads,Kuntz2016bounding,Magron2017invariant,Korda2021invariant,bose2007duality,bose2005minimum}). Here we closely follow~\cite{Korda2021invariant} and refer readers to this work for further details and convergence proofs.

\subsection{Background and Definitions}

Let $X\subset\mathbb{R}^n$ be a compact set and, given a continuous function $f:X \to X$, consider the discrete-time dynamical system governed by
\begin{equation}\label{DiscreteDS}
	x_{n+1} = f(x_n).
\end{equation}
A nonnegative Borel measure $\mu$ supported on $X$ is \emph{invariant} for the mapping $f$ if  
\begin{equation}\label{InvMeas}
	\mu(f^{-1}(A)) = \mu(A)
\end{equation}
for all Borel measurable sets $A \subseteq X$. Equivalently, $\mu$ is invariant if
\begin{equation}\label{InvMeasInt}
	\int_X [g\circ f - g]\mathrm{d}\mu = 0
\end{equation} 
for all functions $g \in C(X)$. We assume throughout that $\mu$ is in the space $\Pr(X)$ of probability measures on $X$, meaning $\mu(X) = 1$. The set of all invariant probability measures for the mapping $f$  is convex and its extreme points are the ergodic measures.

Condition \cref{InvMeasInt} has a deep connection to the Koopman operator and its adjoint, the Perron--Frobenius operator. The Koopman operator associated to $f$ is the linear operator $\mathcal{K}:C(X) \to C(X)$ encoding the composition of continuous functions with $f$, i.e.,
\begin{equation}
	\mathcal{K}g := g \circ f.
\end{equation} 
The Perron--Frobenius operator, instead, is a linear operator $\mathcal{P}$ on the space of measures that associates a measure $\mu$ to its pushforward by $f$, i.e.,
\begin{equation}\label{PerronFrobenius}
	(\mathcal{P}\mu)(A) = \mu(f^{-1}(A)).
\end{equation}
It should be clear that \cref{InvMeasInt} can be written in terms of the Koopman operator as
\begin{equation}\label{InvMeasInt2}
	\int_X [\mathcal{K}g- g]\mathrm{d}\mu = 0,	
\end{equation}
while \cref{InvMeas} can be expressed as $\mathcal{P}\mu = \mu$. In other words, invariant measures are eigenfunctions of the Perron--Frobenius operator with eigenvalue $1$.

\subsection{Convex Computation of Invariant Measures}

Invariant measures for a known polynomial map $f$ can be approximated using a convex optimization framework developed in \cite{Korda2021invariant}, which we now review in detail as it forms the basis of our data-driven method. 

Fix an integer $d\geq 1$. For this section, and this section only, we assume the following.

\begin{assumption}\label{ass:poly-f}
	The function $f:X\to X$ is a polynomial of total degree $\leq d$.
\end{assumption}

\noindent
We also introduce the following assumption about the set $X$.

\begin{assumption}\label{ass:semiagebraic-X}
	$X \subset \mathbb{R}^n$ is a compact set such that:
	\begin{enumerate}[topsep=1ex,leftmargin=2\parindent, labelsep=*]
		\item $X = \{x\in\mathbb{R}^n|\ \sigma_1(x) \geq 0,\ldots,\sigma_J(x) \geq 0\}$ for finitely many polynomials $\sigma_1,\dots,\sigma_J$ of degree $\leq d$.
		\item There exists an integer $R$ and polynomials $s_0,\ldots,s_J$ that are sums of squares of other polynomials such that $R^2-|x|^2 = s_0(x) + s_1(x)\sigma_1(x) + \cdots + s_J(x)\sigma_J(x)$.
	\end{enumerate} 
\end{assumption}

\noindent
Sets satisfying the first condition in \cref{ass:semiagebraic-X} are called {\em basic semialgebraic}. Examples include boxes and balls, but the definition is sufficiently general to include more complicated shapes, both convex and non-convex. The second condition in \cref{ass:semiagebraic-X}, instead, is a technical requirement for the convergence results stated in \cref{th:convergence} below. This requirement is mild when $X$ is compact, because one can always add the inequality $R^2-|x|^2 \geq 0$ to the definition of $X$ with $R$ sufficiently large not to change the set.

\subsubsection{Moment-based Characterization of Invariance}
\label{ss:inv-via-moments}
Since $X$ is compact, any measure $\mu$ on $X$ is uniquely determined by its moments 
\begin{equation}
	y_\alpha = \int_X x^\alpha \mathrm{d}\mu,
\end{equation}
where $x^\alpha = x_1^{\alpha_1}\cdots x_n^{\alpha_n}$ is the monomial in $x$ with exponent $\alpha = (\alpha_1,\dots,\alpha_n) \in \mathbb{N}_0^n$. We write $\vert \alpha \vert=\alpha_1+\cdots+\alpha_n$ for the degree of $x^\alpha$. The first step to study invariant measures via convex optimization is to encode the invariance of a measure using linear constraints on its moments.  

Since polynomials are dense in $C(X)$, we conclude from \cref{InvMeasInt} that a measure $\mu$ supported on $X$ is invariant if and only if
\begin{equation}\label{InvMeasMoments}
	\int_X \left[ (f(x))^\alpha - x^\alpha \right]\mathrm{d}\mu = 0
\end{equation} 
for every $\alpha \in \mathbb{N}_0^n$.
\Cref{ass:poly-f} implies that $(f)^\alpha$ is a polynomial of degree $\vert \alpha \vert d$. Thus, for each exponent $\alpha$, condition \cref{InvMeasMoments} provides a {\em linear} constraint involving the (finitely many) moments of $\mu$ of degree up to $\vert \alpha \vert d$. Letting $\alpha$ vary over $\mathbb{N}_0^n$, we conclude that $\mu$ is an invariant measure if and only if the corresponding moment sequence $y = \{y_\alpha\}_{\alpha \in \mathbb{N}_0^n}$ satisfies an infinite set of linear equality constraints. We write such constraints as
\begin{equation}\label{InvMeasKernel}
	\mathcal{A}y = 0,
\end{equation}  
where $\mathcal{A}:\mathbb{R}^\infty \to \mathbb{R}^\infty$ is an infinite-dimensional linear operator. We also require that 
\begin{equation}\label{InvMeasNormalize}
	y_0 = 1
\end{equation}
because we focus on probability measures.

\subsubsection{Optimization over Invariant Measures}
\label{sec:Objectives}

We have seen above that finding an invariant measure is tantamount to finding a sequence $y \in \mathbb{R}^\infty$ whose elements are moments of a measure and satisfy \cref{InvMeasKernel,InvMeasNormalize}. These conditions have multiple solutions if the dynamical system \cref{DiscreteDS} admits multiple invariant measures. To try and single out an invariant measure with specific properties, one can optimize a cost function $F(y)$. Following \cite{Korda2021invariant}, we let $F(y)$ be a convex and lower semicontinuous function that depends only on the first $\binom{n+d}{d}$ entries of $y$, which correspond to the moments of degree up to $d$.
This leads to the convex optimization problem
\begin{equation}\label{MeasureLP}
	\min_{\substack{y \in \mathbb{R}^\infty\\ \mu \in \Pr(X)}} \ F(y) \quad \text{such that}\quad \begin{cases}
		\mathcal{A}y = 0, \\
		y_0 = 1, \\
		y_\alpha = \int_X x^\alpha \,\mathrm{d}\mu \quad 
		\forall \alpha\in\mathbb{N}_0^n.
	\end{cases}
\end{equation} 

Examples of useful cost functions are discussed in \cite{Korda2021invariant}. Here, we will consider only cost functions that promote the identification of physical measures and ergodic measures.

\paragraph{Physical Measures.} A measure $\mu$ is said to be \emph{physical} if
\begin{equation}
	\int_X g \, \mathrm{d}\mu = \lim_{N \to \infty} \frac{1}{N}\sum_{n = 0}^{N-1} g(f^n(x))
\end{equation}
for Lebesgue-almost-every point $x\in X$ and every $g\in C(X)$.
Taking $g(x) = x^\alpha$ and fixed $N \gg 1$, the moments of a physical measure $\mu$ satisfy
\begin{equation}\label{ApproxMoment}
	y_\alpha \approx \frac{1}{N}\sum_{n = 0}^{N-1} g(f^n(x_0)) =: \tilde{y}_\alpha
\end{equation}
for Lebesgue-almost-every point $x_0 \in X$. Then, to identify $\mu$, one can compute $\tilde{y}_\alpha$ for some set $\mathcal{I}$ of exponents $\alpha$ with degree $\vert \alpha \vert \leq d$ and solve \cref{MeasureLP} with cost function
\begin{equation}\label{e:phys-meas-cost}
	F(y) = \sum_{\alpha \in \mathcal{I}} (y_\alpha - \tilde{y}_\alpha)^2.
\end{equation}

\paragraph{Ergodic Measures.} An invariant probability measure $\mu \in \Pr(X)$  for \cref{DiscreteDS} is said to be \emph{ergodic} if every invariant set has either full or zero measure, that is, $f^{-1}(S)=S$ implies $\mu(S)\in\{0,1\}$.  Ergodic measures are extreme points of the set of all invariant probability measures of $f$. Since the minima of linear programs are attained at extreme points, one can target an ergodic measure by solving \cref{MeasureLP} with a {\em linear} cost function $F$.

\subsubsection{Approximation via Semidefinite Programming}
\label{ss:sdp-relax}

The minimization problem in \cref{MeasureLP} is convex and returns an optimal invariant measure for a given cost function, but it is infinite-dimensional and thus computationally intractable. We now explain how to approximate it with a hierarchy of semidefinite programs (SDP), which are tractable optimization problems with positive semidefiniteness constraints on matrices that depend linearly on the optimization variables.

The first step is to replace the infinite sequence $y$ with a finite vector, also denoted $y$ with a slight abuse of notation. Fix an integer $k \geq 1$. Rather than the infinite set of equality constraints in~\cref{InvMeasKernel}, which are obtained by imposing \cref{InvMeasMoments} for all exponents $\alpha\in\mathbb{N}_0^n$, we enforce this condition only for the $\binom{n+k}{k}$ exponents with degree $\vert \alpha \vert \leq k$. This gives $\binom{n+k}{k}$ constraints for the vector $y$ of $\binom{n+kd}{kd}$ moments of $\mu$ of degree up to $kd$, written compactly as
\begin{equation}\label{InvMeasKernelFinite}
	A_ky = 0
\end{equation} 
for an explicitly computable matrix $A_k \in \mathbb{R}^{\binom{n + k}{k}\times \binom{n + dk}{dk}}$. Observe that $A_k$ has a high-dimensional kernel when $k$ is large and $f$ is nonlinear ($d\geq 1$). 

Next, we seek to identify elements in the kernel of $A_k$ that correspond to truncated moment sequences of probability measures on $X$. Unfortunately, this problem is intractable in general. Following a standard relaxation strategy for moment problems (see \cite{laurent2009sums,deKlerk2018} and references therein), we therefore limit ourselves to identifying kernel elements $y$ that make the \emph{moment matrices} $M(\sigma_j y)$, $j = 0,1,\dots, J$, positive definite. Setting $\sigma_0=1$, these matrices are defined by
\begin{equation}\label{e:moment-matrices-def}
	M(\sigma_j y) = \int_X \sigma_j(x)p_{\gamma_j}(x)p_{\gamma_j}(x)^\intercal \mathrm{d}\mu, \qquad j = 0,1,\dots,J,
\end{equation}	
where $\gamma_j = \lfloor (dk - \deg \sigma_j) / 2 \rfloor$ and $p_\gamma(x)$ is the vector of monomials in $x$ up to degree $\gamma$. The square matrix $M(\sigma_j y)$ is therefore of size $\binom{n + \gamma_j}{\gamma_j} \times \binom{n + \gamma_j}{\gamma_j}$. We also impose the normalization condition $y_0=1$ to arrive at the following tractable \emph{semidefinite program} (SDP):
\begin{equation}\label{MeasureSDP}
	\min_{y \in \mathbb{R}^{\binom{n + k}{k}}} \ F(y) \quad \text{such that}\quad \begin{cases}
		A_k y = 0 \\
		y_0 = 1 \\
		M(\sigma_jy) \succeq 0, & j = 0,1,\dots,J.
	\end{cases}
\end{equation} 

This SDP is a relaxation of \cref{MeasureLP} in the sense that its feasible set contains all moment vectors of all invariant measures for the discrete-time system \cref{DiscreteDS}. In general, however, it also contains elements that are not moment vectors, so the optimal $y$ need not be a vector of moments of an invariant measure. Nevertheless, one can prove that optimal solutions to \cref{MeasureSDP} converge (up to a subsequence) to moment sequences of optimal invariant measures for \cref{MeasureLP}. 

\begin{theorem}[{\cite[Theorem~2]{Korda2021invariant}}]\label{th:convergence}
	Suppose \cref{ass:poly-f,ass:semiagebraic-X} hold. Let $y^k \in \mathbb{R}^{\binom{n+kd}{kd}}$ be the optimal solution of \cref{MeasureSDP} for fixed $k \in \mathbb{N}$, and let $\mu$ be optimal for \cref{MeasureLP}. There exists an increasing sequence $\{k_j\}_{j \geq 1} \subset \mathbb{N}$ such that
	\begin{equation*}
		\lim_{j \to \infty} y^{k_j}_\alpha = \int_X x^\alpha \, \mathrm{d}\mu \qquad \forall \alpha \in \mathbb{N}_0^n.
	\end{equation*}
	In particular, if \cref{MeasureLP} has a unique optimizer $\mu$, then $y_\alpha^k \to \int x^\alpha \,\mathrm{d}\mu$ as $k\to \infty$.
\end{theorem}

Finally, we remark that the SDP \cref{MeasureSDP} can be formulated and solved even when the set $X$ does not satisfy the second condition in \cref{ass:semiagebraic-X} and, in particular, when $X=\mathbb{R}^n$ is the full space. In such cases, however, it is currently not know if the convergence results from \cref{th:convergence} always apply. Fortunately, convergence if often observed in examples (see, e.g., \cite{Chernyshenko2014rsta,Fantuzzi2016siads,goluskin2018bounding,goluskin2019kuramoto}).

\subsubsection{Recovery of an Approximate Invariant Measure}
\label{ss:recovery}
Solutions of the SDP \cref{MeasureSDP} for increasing $k$ recover, possibly after passing to a subsequence, the moments of an optimal invariant measure $\mu$ for \cref{MeasureLP}. We now review how to approximate $\mu$ using a sequence of \emph{signed} measures $\mu_k$ constructed from the optimal solutions $y^k$ of \cref{MeasureSDP}. For notational ease, we assume we have already passed (without relabelling) to a convergent subsequence of optimal solutions as in \cref{th:convergence}.

First, fix a finite measure $\pi$ (not necessarily a probability measure) whose support is contained in $X$ and such that the associated moment matrix
\begin{equation}
	M_\pi := \int_X p_{k}(x)p_{k}(x)^\intercal \, \mathrm{d}\pi
\end{equation}
is computable for every integer $k$. We assume $\pi$ is supported on a full-dimensional or non-algebraic subset of $X$, so as to ensure $M_\pi$ is invertible. In the examples of \cref{s:examples} we will take $\pi$ to be the Lebesgue measure on $X$, but here we work with general $\pi$ to highlight that other choices are possible.

Next, consider the degree-$k$ polynomial
\begin{equation}
	\rho_k(x) = p_{k}(x)^\intercal \, M_\pi^{-1}y^k
\end{equation}
and define a $\pi$-absolutely-continuous signed measure $\mu_k$ via
\begin{equation}
	\mu_k(A) = \int_A \rho_k(x) \, \mathrm{d}\pi \qquad \forall A \subset X.
\end{equation}
This measure satisfies $\mu_k(X)=1$ by construction. Note, however, that $\mu_k$ is a \emph{signed} measure in general because $\rho_k$ need not be nonnegative on the support of $\pi$. Nevertheless, a straightforward extension of arguments in \cite{Korda2021invariant} shows that $\mu_k$ converges to an optimal invariant measure $\mu$ for \cref{MeasureLP} in the weak-star topology of $\Pr(X)$, meaning that
\begin{equation}\label{e:weak-conv}
	\lim_{k \to \infty} \int_X g(x)\rho_k(x) \, \mathrm{d}\pi = \int_X g(x)\,\mathrm{d}\mu
\end{equation}
for every continuous function $g:X \to \mathbb{R}$. In other words, as the polynomial degree $k$ increases, the signed measures $\mu_k$ enable one to approximate as accurately as desired the expectation with respect to the invariant measure $\mu$ of any fixed function $g$. Moreover, \cref{e:weak-conv} implies that if $\mu$ is $\pi$-absolutely-continuous with density $\rho$, then the signed densities $\rho_k$ converge to $\rho$ in the weak topology of $L^1(\pi)$.

\section{Discovering Extremal Invariant Measures from Data}\label{sec:InvMeasData}

In the convex optimization approach to identifying invariant measures reviewed in \cref{sec:InvMeas}, the key ingredient is the matrix $A_k$. This matrix is the only component of the optimization problem \cref{MeasureSDP} that encodes the system dynamics and it is built by considering the action of the Koopman operator, $\mathcal{K}g = g \circ f$, on monomials $g(x)=x^\alpha$ of degree $\vert \alpha \vert\leq k$. Recent advances in data-driven analysis of dynamical systems make it possible to approximate the action of the Koopman operator on finite-dimensional spaces of functions, called \emph{observables}, using only data gathered from the system \cite{eDMD,klus2016towards,klus2020data}. Here, we show how to use these data-driven methods together with the computational approach from \cref{sec:InvMeas} to discover invariant measures entirely from data. We henceforth assume we do not know the map $f$ in \cref{DiscreteDS}, but we have a set of $m$ data snapshots $\{ (x_i,z_i) \}_{i = 1}^m \subset X \times X$ where $z_i = f(x_i)$. We do not assume that $f$ is polynomial, but still require $X$ to be a basic semialgebraic set according to the first condition in \cref{ass:semiagebraic-X}.

Fix integers $k$ and $\ell$ with $1 \leq k \leq \ell$, and let $p_k(x)$ and $p_\ell(x)$ be the vectors of monomials in $x$ up to degrees $k$ and $\ell$, respectively. Let $k_x = \binom{n + k}{k}$ and $\ell_x = \binom{n + \ell}{\ell}$ denote the size of these vectors, respectively, and consider the matrices 
\begin{align}\label{EDMDmatrices}
	P &= \begin{bmatrix}
		| &  & | \\
		p_\ell(x_1) & \cdots & p_\ell(x_m) \\
		| &  & | 
	\end{bmatrix} \in \mathbb{R}^{\ell_x \times m}
	\intertext{and}
	Q &= \begin{bmatrix}
		| &  & | \\
		p_k(z_1) & \cdots & p_k(z_m) \\
		| &  & | 
	\end{bmatrix} \in \mathbb{R}^{k_x\times m}.
\end{align} 
Extended dynamic mode decomposition (EDMD) seeks to approximate the action of the Koopman operator on $\mathrm{span}\{p_k(x)\}$ using a linear operator $\mathcal{K}_{k\ell}:\mathrm{span}\{p_k(x)\}\to \mathrm{span}\{p_\ell(x)\}$ called the {\em approximate Koopman operator}. The matrix $K$ representing this operator is obtained by minimizing the Frobenius matrix norm $\|Q - KP\|_F$.
The optimal solution is
\begin{equation}
	K_{k\ell} := QP^\dagger \in \mathbb{R}^{k_x \times \ell_x}
\end{equation}
where the superscript $\dagger$ denotes the Moore--Penrose pseudoinverse. The approximate Koopman operator then acts on elements $g(x) = c \cdot p_k(x)$ in $\mathrm{span}\{p_k(x)\}$ by 
\begin{equation}
	\mathcal{K}_{k\ell} \,g(x) := c\cdot K_{k\ell} \,p_\ell(x).
\end{equation} 

Next, let $\Theta_{k\ell} \in \mathbb{R}^{k_x \times \ell_x}$ be the matrix that extracts the monomials in $p_k(x)$ from $p_\ell(x)$, meaning that $p_k(x) = \Theta_{k\ell} p_\ell(x)$. Define the matrix
\begin{equation}
	L_{k\ell} := K_{k\ell} - \Theta_{k\ell} \in \mathbb{R}^{k_x \times \ell_x}.
\end{equation}
The {\em approximate Lie derivative} $\mathcal{L}_{k\ell}$ is a linear operator from $\mathrm{span}\{p_k(x)\}$ to $\mathrm{span}\{p_\ell(x)\}$ that acts on any function $g(x) = c \cdot p_k(x)$ in $\mathrm{span}\{p_k(x)\}$ by 
\begin{equation}\label{e:approx-lie}
	\mathcal{L}_{k\ell}\,g(x) = c\cdot L_{k\ell} \,p_\ell(x) = \mathcal{K}_{k\ell}\,g(x) - g(x).
\end{equation} 
Using this definition, condition \cref{InvMeasInt2} for the invariance of a measure $\mu$ can be approximated by
\begin{equation}\label{InvMeasData}
	\int_X \mathcal{L}_{k\ell} g \, \mathrm{d}\mu = 0 \qquad \forall g \in \mathrm{span}\{p_k(x)\}.
\end{equation} 
Note that this approximation is built purely from the data snapshots and requires no knowledge of the map $f$ driving the dynamics.

Condition \cref{InvMeasData} is equivalent to requiring that the vector $y \in \mathbb{R}^{\ell_x}$ of moments of $\mu$ with degree up to $k$ satisfies $L_{k\ell} \, y = 0$.  This finite set of constraints is a data-driven approximation of the condition $A_ky = 0$ in \cref{MeasureSDP}. This suggests that an approximation to an invariant measure $\mu$ for \cref{DiscreteDS} may be discovered from the data snapshots by first solving the SDP     
\begin{equation}\label{MeasureSDPData}
	\min_{y \in \mathbb{R}^{\ell_x}} \ F(y) \quad \text{such that}\quad \begin{cases}
		L_{k\ell}\, y = 0 \\
		y_0 = 1 \\
		M(\sigma_jy) \succeq 0, & j = 0,1,\dots,J,
	\end{cases}
\end{equation} 
and then applying the measure recovery strategy outlined in \cref{ss:recovery}.
Once again, observe that the SDP in \cref{MeasureSDPData} is constructed only using data gathered from the system and without requiring any knowledge of the underlying map $f$. In particular, one can construct it and solve it even when the dynamics are not governed by polynomial $f$. 
 
\begin{remark} 
	Contrary to standard practice, we implement EDMD using two different polynomial dictionaries $p_k$ and $p_\ell$. This is useful because the space of degree-$k$ polynomials is not generally invariant under the action of the Koopman operator, so $\mathcal{K}g$ for a polynomial $g$ of degree $k$ can be better approximated using polynomials of degree $\ell \geq k$. If the data snapshots are governed by a polynomial map of known degree $d \geq1$, then it would suffice to set $\ell = d k$ as in \cref{sec:InvMeas}. When the map $f$ is unknown, however, one should implement the method with fixed $k$ and increasingly large values of $\ell$ until good performance is observed. This can be expected given theoretical analysis in \cite{bramburger2023auxiliary}, which under mild conditions guarantees the convergence of the approximate Lie derivative $\mathcal{L}_{k\ell}$ to the true Lie derivative $\mathcal{L}$ acting on $\mathrm{span}\{p_k(x)\}$ as both the number $m$ of data snapshots and the degree $\ell$ tend to infinity. More precise results for the case of unknown polynomial maps are presented in \cref{sec:Convergence}.
\end{remark}

\begin{remark}
	Above, as well as in \cref{sec:InvMeas}, one can replace the monomial vectors $p_k(x)$ and $p_\ell(x)$ with any bases for the spaces of polynomials in $x$ of degree $k$ and $\ell$, respectively. Moments are then expressed with respect to these bases. Choosing a good polynomial basis is key to ensure the SDPs \cref{MeasureSDP,MeasureSDPData} are well conditioned. For dynamical systems in one dimension, for instance, it is usually convenient to use the Chebyshev basis. Unfortunately, we are not aware of a basis selection strategy that works well in general.
\end{remark}

\section{Convergence for Polynomial Maps}\label{sec:Convergence}

We now investigate the convergence of the data-driven approach introduced in the previous section to the model-based approach of \cite{Korda2021invariant} when the data points $\{(x_i,z_i)\}_{i=1}^m$ are generated by a polynomial map $f$ of degree $d$. We are interested in particular in clarifying how solutions of the data-driven SDP \cref{MeasureSDPData} behave as the number $m$ of datapoints becomes infinite.

Loosely speaking, one expects that optimal solutions of our data-driven SDP should converge to optimal solutions to the model-based SDP \cref{MeasureSDP}. However, these two problems have a different number of variables unless the polynomial degrees $k$ and $\ell$ satisfy $\ell = dk$, so they cannot be compared directly. Instead, under suitable conditions, we will show that for every $\ell \geq dk$ the optimal solutions of the data-driven SDP \cref{MeasureSDPData} converge (up to subsequences) to an optimal solution of the "limit" problem
\begin{equation}\label{MeasureSDPExtended}
	\min_{y \in \mathbb{R}^{\ell_x}} \ F(y) \quad \text{such that}\quad \begin{cases}
		A_{k\ell}\, y = 0 \\
		y_0 = 1 \\
		M(\sigma_jy) \succeq 0, & j = 0,1,\dots,J.
	\end{cases}
\end{equation}
Here $A_{k\ell}$ is the unique $k_x \times \ell_x$ matrix whose first $\binom{n + dk}{dk}$ columns are the columns of the matrix $A_k$ in \cref{MeasureSDP}, and whose remaining columns are zero. This new SDP has the same number of variables as the data-driven SDP \cref{MeasureSDPData} and, for every $\ell \geq dk$, it enjoys the same convergence properties as \cref{MeasureSDP} as $k\to \infty$. Precisely, the following extension of \cref{th:convergence} holds. 

\begin{theorem}\label{th:convergence-extended}
	Suppose \cref{ass:poly-f,ass:semiagebraic-X} hold and assume $\ell \geq dk$. Let $y^k \in \mathbb{R}^{\ell_x}$ be the optimal solution of \cref{MeasureSDPExtended} for fixed $k \in \mathbb{N}$, and let $\mu \in \Pr(X)$ be optimal for \cref{MeasureLP}. There exists an increasing sequence $\{k_j\}_{j \geq 1} \subset \mathbb{N}$ such that $y^{k_j}_\alpha \to \int_X x^\alpha \, \mathrm{d}\mu$ as $j\to \infty$ for all $\alpha\in \mathbb{N}_0^n$.
\end{theorem}

The proof of this statement is the same as that of \cite[Theorem~2]{Korda2021invariant}, so we omit the details. Briefly, \cref{ass:semiagebraic-X} ensures that the sequence $\{y_\alpha^k\}_{k \geq 1}$ is uniformly bounded in $k$ for every $\alpha \in \mathbb{N}_0^n$, so by a diagonalization argument one can extract a subsequence $\{k_j\}_{j \geq 1}$ and real numbers $y_\alpha^*$ such that $y_\alpha^{k_j} \to y_\alpha^*$ for every $\alpha \in \mathbb{N}_0^n$. Moreover, always by \cref{ass:semiagebraic-X}, there exists a probability measure $\mu \in \Pr(X)$ such that $y_\alpha^* = \int_X x^\alpha \, \mathrm{d}\mu$ for every $\alpha$. The equality constraints in \cref{MeasureSDPExtended} ensure that $\mu$ is invariant for the map $f$.

The rest of this section is dedicated to proving that the data-driven SDP \cref{MeasureSDPData} reduces almost surely to problem \cref{MeasureSDPExtended} as the number $m$ of data points increases. (The meaning of the quantifier "almost surely" is made precise in the next subsection.) We start in \cref{ss:analysis-assumptions} by stating key assumptions and background results that enable our convergence analysis. In \cref{ss:analysis-feasibility} we establish that the data-driven SDP \cref{MeasureSDPData} is feasible almost surely if the number $m$ of data points is large enough. We show that its optimal solutions converge almost surely to those of \cref{MeasureSDPExtended} as $m\to\infty$ in \cref{ss:analysis-convergence}, and conclude in \cref{ss:conv-failure} with further comments on our analysis.

\subsection{Preliminaries}
\label{ss:analysis-assumptions}

In order to study how solutions of the data-driven SDP \cref{MeasureSDPData} behave with increasing amount of data, we need to introduce some assumptions on the data collection process. The first one imposes some technical restrictions on the polynomial map $f:X \to X$ from which the data is collected.

\begin{assumption}\label{ass:f-assumptions}
	The data points $\{(x_i,z_i)\}_{i=1}^m$ satisfy $z_i=f(x_i)$ where $f:X\to X$ is a polynomial of degree $d$ such that:
	\begin{enumerate}[topsep=1ex,leftmargin=2\parindent, labelsep=*]
		\item For every integer $p$ and every choice of multi-indices $\alpha_1,\ldots,\alpha_p \in \mathbb{N}_0^n$, the polynomials $[f(x)]^{\alpha_1} - x^{\alpha_1},\,\ldots,\, [f(x)]^{\alpha_p} - x^{\alpha_p}$ and $1$ are linearly independent.
		\item There exists an invariant probability measure $\mu \in \Pr(X)$ for $f$ whose support is not an algebraic set, meaning that no nonzero polynomial vanishes on it.
	\end{enumerate}
\end{assumption}

Next, we assume that the data points $x_i$ are generated by a random sampling process. The sequence of datasets $\{(x_i,z_i)\}_{i=1}^m$ is then also a random variable, and we say that a statement holds \emph{almost surely} as $m\to \infty$ if it holds for almost every sequence of increasingly large  datasets. 
\begin{assumption}\label{ass:sampling}
	Each data point $x_i$ is an independent realization of a random variable whose distribution is a probability measure $\nu \in \Pr(X)$. 
\end{assumption}
\noindent
The assumption of independence can in fact be relaxed: the points $x_i$ only need to be sampled from $\nu \in \Pr(X)$ in such a way that, as $m\to\infty$, the empirical averages of polynomial functions over the dataset converge almost surely to the averages calculated with respect to the sampling measure $\nu$ \cite[Assumption~4.2]{bramburger2023auxiliary}. This is true, for instance, when the datapoints $x_i$ are collected from an ergodic trajectory, as we do for the computational examples of \Cref{s:examples}.

Irrespective of how exactly the data is sampled, we assume that no nonzero polynomial vanishes on the support of the sampling measure $\nu$.
\begin{assumption}\label{ass:support}
	If $g$ is a polynomial such that $\int_X g(x) \,\mathrm{d}\nu=0$, then $g(x)\equiv 0$.
\end{assumption}
\noindent
This assumption ensures that two polynomials are equal if they take the same values on the support of $\nu$. In particular, if a sequence of polynomials $\{g_j\}$ converges to a polynomial $g$ as $j \to \infty$ in the Lebesgue space $L^2(\nu)$ of functions that are square integrable with respect to $\nu$, then they converge pointwise on $\mathbb{R}^n$ and uniformly on every compact set. 

Finally, we recall the following version of a more general convergence result for approximate Lie derivatives from \cite{bramburger2023auxiliary}, where $\mathcal{P}_\ell(u)$ denotes the $L^2(\nu)$-orthogonal projection of a function $u$ onto the space of degree-$\ell$ polynomials. This is a key ingredient in our convergence analysis.

\begin{theorem}[{\cite[Theorem 4.1]{bramburger2023auxiliary}}]\label{th:conv-L2}
	For every degree-$k$ polynomial $g$ and every $\ell \geq k$, the approximate Lie derivative $\mathcal{L}_{k\ell} g$ in \cref{e:approx-lie} converges to $\mathcal{P}_\ell (\mathcal{K} g - g)$ almost surely as $m \to \infty$.
\end{theorem}

\noindent
We stress that this theorem only guarantees convergence as the number of datapoints $m$ grows without bound, but provides no convergence rates. Analysis \cite{zhang2023quantitative} and computational evidence \cite{bramburger2023auxiliary} reveal that if the data is generated by random sampling, then approximation errors decay like $1/\sqrt{m}$. However, this rate could be even slower if data is collected from a single trajectory because time averages can be made to converge to their ergodic value arbitrarily slowly~\cite{krengel1978speed}.

\subsection{Almost-sure Strict Feasibility}
\label{ss:analysis-feasibility}
Having stated all of our assumptions, we now show that if the number $m$ of data points and the polynomial degree $\ell$ are large enough, then the data-driven SDP \cref{MeasureSDPData} is almost surely \emph{strictly feasible}. This means that there exists a vector $y$ satisfying the constraints of \cref{MeasureSDPData} for which the moment matrices $M(\sigma_j y)$ are positive definite.

\begin{proposition}\label{prop:feasibility}
	Suppose \cref{ass:f-assumptions,,ass:sampling,,ass:support} hold. For every $k$ and every $\ell \geq dk$, there almost surely exists $m_0 \in \mathbb{N}$ such that the SDP \cref{MeasureSDPData} is strictly feasible for all $m \geq m_0$.
\end{proposition}

\begin{proof}
	Let $A_{k\ell}$ be the $k_x \times \ell_x$ matrix appearing in problem \cref{MeasureSDPExtended}. Consider the vector $\bar{y} \in \mathbb{R}^{\ell_x}$ of degree-$\ell$ moments of the invariant measure $\mu$ in the second part of \cref{ass:f-assumptions}. This assumption ensures that $A_{k\ell} \bar{y} =0$, $\bar{y}_0 = 1$, and $M(\sigma_j \bar{y})\succ 0$, so $\bar{y}$ is a strictly feasible point of~\cref{MeasureSDPExtended}. 
	
	Next, let $e \in \mathbb{R}^{\ell_x}$ be the vector such that $1=e \cdot p_\ell(x)$ and set
	\begin{equation}
		B = \begin{bmatrix} A_{k\ell} \\ e^\top \end{bmatrix} \in \mathbb{R}^{(k_x +1)\times \ell_x}\qquad
		B_m = \begin{bmatrix} L_{k\ell} \\ e^\top \end{bmatrix} \in \mathbb{R}^{(k_x +1)\times \ell_x}\qquad
		c = \begin{bmatrix} 0 \\ 1 \end{bmatrix} \in \mathbb{R}^{k_x +1},
	\end{equation}
	so the equality constraints of the SDPs \cref{MeasureSDPExtended,MeasureSDPData} can be written as $B_m y = c$ and $By=c$, respectively. Note that $B_m$ depends on the number of data points because so does $L_{k\ell}$, while $B$ does not. For each $m$, the vector
	\begin{equation}
		y^m := ( I - B_m^\dagger B_m) \bar{y} + B_m^\dagger c
	\end{equation}
	satisfies $B_m y^m = c$. We need to prove that, almost surely, there exists an integer $m_0>0$ such that $M(\sigma_j y^m)\succ 0$ when $m \geq m_0$, meaning that $y^m$ is a strictly feasible point for~\cref{MeasureSDPData}.

	Since the matrices $M(\sigma_j \bar{y})$ are positive definite and their eigenvalues are continuous with respect to perturbations in $\bar{y}$, it is enough to show that $y^m$ converges to $\bar{y}$ almost surely as $m \to \infty$. To see this, we use the definition of $y^m$ and the identity $c=B\bar{y}$ to estimate
	\begin{equation}\label{e:proof-estimate-1}
		\begin{aligned}
			\|y^m - \bar{y}\| 
			&= \| B_m^\dagger B_m \bar{y} - B_m^\dagger c\| \\ 
			&= \| B_m^\dagger \left( B_m \bar{y} - B \bar{y} \right)\| \\
			&\leq \| B_m^\dagger \| \|B_m - B \| \| \bar{y} \|.
		\end{aligned}
	\end{equation}
	We now claim that $B_m$ converges almost surely to $B$ and that the norm $\| B_m^\dagger \|$ is almost surely uniformly bounded in $m$, so $\|y^m - \bar{y}\|$ tends to zero almost surely as desired.

	For the first claim, fix an arbitrary degree-$k$ polynomial $g(x)=w \cdot p_k(x)$, where $w \in \mathbb{R}^{k_x}$. By \cref{th:conv-L2}, the approximate Lie derivative $\mathcal{L}_{k\ell} g = w \cdot L_{k\ell} p_\ell(x)$ converges almost surely to $\mathcal{P}_\ell (\mathcal{K} g - g) = \mathcal{P}_\ell (g \circ f - g)$ as $m$ is raised. Since $g \circ f - g= w \cdot A_{k\ell} p_\ell$ is a polynomial of degree at most $dk \leq \ell$, we must have $\mathcal{P}_\ell (g \circ f - g) = w \cdot A_{k\ell} p_\ell$ on the support of the data sampling measure $\nu$. Then, \cref{ass:support} guarantees that $w \cdot L_{k\ell} p_\ell(x)$ converges to $w \cdot A_{k\ell} p_\ell(x)$ pointwise in $x$. This is true for all vectors $w \in \mathbb{R}^{k_x}$, so $L_{k\ell}$ must converge almost surely to $A_{k\ell}$ as $m\to\infty$. This, in turn, proves our claim that $B_m \to B$ almost surely as $m\to\infty$.
	
	To see that $\|B_m^\dagger\|$ is almost surely uniformly bounded in $m$, instead, note that the first condition in \cref{ass:f-assumptions} implies that the matrix $B$ has full row rank. This, combined with the almost sure convergence $B_m \to B$ and the sequential lower semicontinuity of the rank function, ensures that $B_m$ has a full row rank almost surely for large enough $m$. Since pseudoinversion is continuous on constant rank sequences \cite{Stewart1969}, we obtain that $B_m^\dagger \to B^\dagger$ almost surely as $m\to \infty$, so $\|B_m^\dagger\|$ is almost surely uniformly bounded.
\end{proof}

\subsection{Convergence with Infinite Data}
\label{ss:analysis-convergence}

We now establish our main convergence result: for every sufficiently large $\ell$, the sequence of optimal solutions of the SDP \cref{MeasureSDPData} obtained for increasing $m$ has almost surely a subsequence that converges to an optimal solution of the model-based SDP \cref{MeasureSDPExtended}. This result can be bootstrapped to the convergence results from \cref{th:convergence-extended} to conclude that if one takes $\ell \geq dk$ and lets the parameters $m$ and $k$ tend to infinity, then our data-driven approach computes moments of invariant measures for the polynomial map $f$ generating the data.
Observe that the  parameters $m$ and $k$ must be increased in this precise order because one must obtain~\cref{MeasureSDPExtended} as the infinite-data limit of~\cref{MeasureSDPData} before one can let $k\to \infty$ and apply \cref{th:convergence-extended}.

\begin{theorem}
	Suppose \cref{ass:semiagebraic-X,,ass:f-assumptions,,ass:sampling,,ass:support} hold. Let the cost function $F$ in \cref{MeasureSDPData,MeasureSDPExtended} be continuous. For every $k$ and every $\ell \geq dk$, the sequence $\{y^m\}_{m\geq 1}$ of optimal solutions of the SDP~\cref{MeasureSDPData} almost surely has a subsequence that converges to an optimal solution of the SDP \cref{MeasureSDPExtended} as $m\to \infty$.
\end{theorem}

\begin{proof}
	Let us first prove that SDP \cref{MeasureSDPExtended} has an optimal solution. The second condition in \cref{ass:f-assumptions} guarantees the existence of a strictly feasible point $y$, while \cref{ass:semiagebraic-X} ensures that the feasible set is compact. Every minimizing sequence must then have a convergent subsequence, whose limit is an optimizer by the continuity of the cost function $F$. The same arguments show that the SDP \cref{MeasureSDPData} has an optimal solution $y^m$ when it is feasible, which is true almost surely for large enough $m$ by \cref{prop:feasibility}. 
	
	Consider now the sequence $\{y^m\}_{m\geq 1} \subset \mathbb{R}^{\ell_x}$ of optimal solutions of \cref{MeasureSDPData}. Since each vector $y^m$ belongs to the set $\{y \in \mathbb{R}^{\ell_x}:\;M(\sigma_j y) \succeq 0 \quad \forall i = 1,\ldots,J\}$, which is compact by \cref{ass:semiagebraic-X}, we can almost surely extract a subsequence of $\{y^m\}_{m\geq 1}$ (not relabelled for simplicity) that converges to a point $y^*$. We claim that $y^*$ is almost surely an optimal solution of \cref{MeasureSDPExtended}.

	To prove this claim, we observe that $y^*$ satisfies the matrix inequalities $M(\sigma_j y) \succeq 0$ for all $i = 1,\ldots,J$ by construction. Moreover, with $B$, $B_m$ and $c$ defined as in the proof of \cref{prop:feasibility}, we have
	\begin{equation}
		\begin{aligned}
			\| B y^* - c \|
			&= \| B y^* - B_m y^m \|\\
			&\leq \| B y^* - B y^m \| + \| (B - B_m) y^m \|\\
			&\leq \| B \| \|y^* - y^m \| + \| B - B_m \| \| y^m \|
		\end{aligned}
	\end{equation}
	because the points $y^m$ satisfy $B_my^m = c$.
	The expression in the last line tends to zero almost surely as $m\to\infty$ because $y^m \to y^*$ and $B_m \to B$ almost surely. Thus, $B y^* = c$ almost surely and we conclude that $y^*$ is almost surely feasible for \cref{MeasureSDPExtended}.

	Next, let $y^\text{opt}$ and $y^{\text{f}}$ be an optimal and a strictly feasible point for~\cref{MeasureSDPExtended}, respectively. Their existence is guaranteed, as explained at the start of the proof. Given arbitrary $\epsilon>0$, let $\delta \in (0,1)$ be small enough that the point
	\begin{equation}
		y^\epsilon := (1-\delta) y^\text{opt} + \delta y^\text{f}
	\end{equation}
	satisfies
	\begin{equation}\label{e:condition-y-eps}
		F(y^\epsilon) \leq F(y^\text{opt}) + \epsilon.
	\end{equation}
	It is not difficult to verify that $y^\epsilon$ is strictly feasible for \cref{MeasureSDPExtended}. Moreover, arguing as in the proof of \cref{prop:feasibility} shows that
	\begin{equation}
		y^{\epsilon,m} := ( I - B_m^\dagger B_m) y^\epsilon + B_m^\dagger c 
	\end{equation}
	converges almost surely to $y^\epsilon$ as $m\to\infty$ and is therefore almost surely feasible for \cref{MeasureSDPData} for all large enough $m$. The continuity of $F$, the optimality of $y^\text{opt}$ and $y^m$, and inequality \cref{e:condition-y-eps} then imply 
	\begin{equation}
		F(y^\text{opt}) \leq F(y^*) =  \lim_{m\to \infty} F(y^m) \leq  \lim_{m\to \infty} F(y^{\epsilon,m}) = F(y^\epsilon) \leq F(y^\text{opt}) + \epsilon.
	\end{equation}
	Since $\epsilon>0$ is arbitrary we conclude that $F(y^\text{opt}) = F(y^*)$, so $y^*$ is optimal for \cref{MeasureSDPExtended}.
\end{proof}

\subsection{Further Discussion}
\label{ss:conv-failure}
We conclude this section by showing that, in general, the convergence of approximate Lie derivative as guaranteed by \cref{th:conv-L2} is by itself not sufficient  to ensure that the data-driven SDP \cref{MeasureSDPData} converges to its exact counterpart \cref{MeasureSDPExtended} as the number $m$ of data points grows.

Consider the map $f(x)=2x-x^2$ on the interval $X=[0,1]$, which we represent using polynomial inequalities as $X=\{x\in\mathbb{R}:\; x-x^2 \geq 0\}$. We seek the invariant measure $\mu$ with the largest possible first moment, $y_1=\int_0^1 x\,\mathrm{d}\mu$. The solution is easily seen to be the atomic measure supported at the fixed point $x=1$, for which $y_1=1$.

We now set up our SDP approach for polynomial degrees $k=1$ and $\ell=2$, which satisfy $\ell\geq dk$ because the polynomial map $f$ we are considering has degree $d=2$. Upon enforcing the constraint $y_0=1$ explicitly, for these parameter choices the model-based SDP \cref{MeasureSDPExtended} reads
\begin{equation}\label{e:example-failure-exact}
	\inf_{y \in \mathbb{R}^2} \ -y_1 
	\quad 
	\text{such that}\quad 
	\begin{cases}
		y_1 - y_2 = 0, \\[1ex]
		\begin{pmatrix}
			1 & y_1 \\ y_1 & y_2
		\end{pmatrix}
		\succeq 0,
		\\[1ex]
		y_1 - y_2 \geq 0.
	\end{cases}
\end{equation}
The first constraint in this problem comes from enforcing the invariance condition $\int_X \mathcal{L} x \, \mathrm{d}\mu = \int_X (x-x^2) \, \mathrm{d}\mu = 0$, which is the only one we must consider given our choice $k=1$. The other two constraints are the moment matrices obtained for $\ell=2$ from the general definition \cref{e:moment-matrices-def} with the polynomials $\sigma_0(x) = 1$ and $\sigma_1(x) = x - x^2$. The optimal solution of \cref{e:example-failure-exact} is easily seen to be $y = (1,1)$ (see the left panel of \cref{f:example-failure}).

\begin{figure}
	\centering
	\includegraphics[width=5cm]{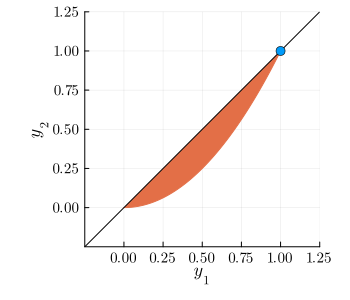}
	\hspace{10pt}
	\includegraphics[width=5cm]{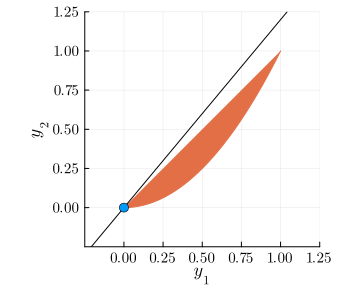}
	\caption{Feasible sets for the SDPs \cref{e:example-failure-exact} (left) and \cref{e:example-failure-data} (right). In both panels, the orange shaded region is the set of points satisfying the (matrix) inequality constraints. Straight black lines indicate the linear spaces $y_1 - y_2=0$ (left panel) and $\alpha_m y_1 - \beta_m y_2 = 0$ for $\alpha_m/\beta_m > 1$ (right panel). Optimal points in both panels are marked by a blue dot.
	}
	\label{f:example-failure}
\end{figure}

The data-driven SDP \cref{MeasureSDPData}, instead, reads
\begin{equation}\label{e:example-failure-data}
	\inf_{y \in \mathbb{R}^2} \ -y_1 
	\quad 
	\text{such that}\quad 
	\begin{cases}
		\alpha_m y_1 - \beta_m y_2 = 0, \\[1ex]
		\begin{pmatrix}
			1 & y_1 \\ y_1 & y_2
		\end{pmatrix}
		\succeq 0,
		\\[1ex]
		y_1 - y_2 \geq 0,
	\end{cases}
\end{equation}
for some coefficients $\alpha_m$ and $\beta_m$ that depend on the size $m$ of the dataset used to approximate the Lie derivative $\mathcal{L}x = x - x^2$ with EDMD. While \cref{th:conv-L2} guarantees that $\alpha_m,\,\beta_m \to 1$ as $m \to \infty$, it gives no additional control on approximation errors, so \emph{a priori} one cannot rule out the possibility that $\alpha_m / \beta_m> 1$ for all $m$. In this case, as illustrated in the right panel of \cref{f:example-failure}, for every $m$ the optimal solution of \cref{e:example-failure-data} would be $y^m=(0,0)$. Thus, $y^m$ would not converge to the optimal solution $y=(1,1)$ of \cref{e:example-failure-exact} despite the convergence of $\alpha_m$ and $\beta_m$.

This potential failure of convergence stems from the lack of strict feasibility for the SDP \cref{e:example-failure-exact}. As shown in the left panel of \cref{f:example-failure}, all feasible points lie on the boundary of the set defined by the (matrix) inequality constraints. Optimal solutions are therefore not robust to arbitrary perturbations to the equality constraints. Such a scenario is not possible in \cref{ss:analysis-convergence}, as the SDP~\cref{MeasureSDPExtended} is strictly feasible by \cref{ass:f-assumptions}(ii). This assumption fails for the present example because the only invariant measures for the map $f(x)=2x-x^2$ on $X=[0,1]$ are the atomic measures supported at the fixed points $x=0$ and $x=1$, which are an algebraic set.

We must however point out that the possible failure of convergence in our present example can be ruled out by an argument different from the proof in \cref{ss:analysis-convergence}. Indeed, since $f$ is a quadratic polynomial map, it can be recovered exactly by sampling it at $m\geq 3$ distinct points $x_1,\ldots,x_m \in [0,1]$. With such a dataset, EDMD recovers the \emph{exact} Lie derivative $\mathcal{L}x$, giving $\alpha_m = \beta_m = 1$ in~\cref{e:example-failure-data}. Since the datapoints $x_1,\ldots,x_m$ are almost surely distinct when sampled according to \cref{ass:sampling,ass:support}, the SDP \cref{e:example-failure-data} coincides with \cref{e:example-failure-exact} almost surely for every $m\geq 3$. The reasoning can in fact be extended to any $n$-dimensional polynomial map of degree $d$ and any choice of the parameters $k$ and $\ell\geq dk$: problem \cref{MeasureSDPData} coincides with \cref{MeasureSDPExtended} if (i) a subset of the data points $x_i$ are \emph{unisolvent} for the space of degree-$\ell$ polynomials, and (ii) the map $f(x)$, or at least the Lie derivatives of degree-$k$ polynomials required by the EDMD step of our approach, are sampled exactly.
This finer result, however, is specific to deterministic polynomial maps. The proofs in \cref{ss:analysis-feasibility,ss:analysis-convergence}, instead, extend to generalizations of our SDP approach to continuous-time and/or stochastic dynamical systems, discussed in the next section. The only requirements for this extension are that the dynamics be governed by polynomial equations, and that approximate Lie derivatives built with EDMD converge to their exact counterparts.

\section{Extensions to Other Dynamical Systems}\label{sec:Extensions}

Our data-driven approach to discovering invariant measures can be applied to dynamical systems beyond those governed by discrete maps. The key observation is that the invariance condition in \cref{sec:InvMeas} involves the \emph{Lie derivative}
\begin{equation}
	\mathcal{L}g:= g \circ f - g
\end{equation}
of the test function $g$, which describes how $g$ changes along trajectories of system \cref{DiscreteDS}. As such, Lie derivatives can be defined for a much broader class of dynamical systems (see \cite[Section~3]{bramburger2023auxiliary}). Classical examples include:
\begin{itemize}[leftmargin=*]
	\item Dynamical systems governed by well-posed autonomous ordinary differential equations. Precisely, let $x(t)$ solve the differential equation $\frac{d}{dt} x = f(x)$. Then, the Lie derivative of any continuously differentiable function $g$ is given by
	\begin{equation}
		\mathcal{L}g =  f\cdot \nabla g.
	\end{equation}
	
	\item Discrete-time systems governed by stochastic maps. Precisely, let $x_{n+1} = f(\omega_n,x_n)$ be a stochastic process where the map $\omega \mapsto f(\omega,\cdot)$ is a random variable from a probability space $(\Omega,\mathcal{F},\pi)$ into the space of maps from $X$ to itself. For such a stochastic process, the Lie derivative acts on continuous functions $g$ by
	\begin{equation}
		\mathcal{L}g(x) = \int_\Omega g(f(\omega,x))\mathrm{d}\pi(\omega) - g(x).
	\end{equation} 

	\item Stochastic processes governed by stochastic differential equations. Precisely, let $X_t$ be the solution of $\mathrm{d} X_t = a(X_t)\mathrm{d} t + b(X_t)\mathrm{d} W_t$, where $W_t$ is a standard $k$-dimensional Wiener process and $a:\mathbb{R}^n\to \mathbb{R}^n$ and $b:\mathbb{R}^n\to \mathbb{R}^{n\times k}$ are given Lipschitz functions. Then, the Lie derivative of any twice-differentiable function $g$ is
	\begin{equation}\label{e:sde-lie-derivative}
		\mathcal{L}g = a \cdot \nabla g + \frac{1}{2}\langle b b^\intercal,\nabla^2 g \rangle,
	\end{equation}
	where $\nabla^2 g$ is the Hessian of $g$ and the angled brackets indicate the Frobenius (element-wise) inner product.  
\end{itemize}

In all of these cases, invariant measures are characterized by the condition that
\begin{equation}
	\int_X \mathcal{L}g \,\mathrm{d}\mu = 0
\end{equation}
for all functions $g$ for which the Lie derivative is defined. In addition, EDMD with polynomial dictionaries can be used to approximate the Lie derivative from data with guaranteed convergence as the amount of data and the size of the polynomial dictionaries increase~\cite{bramburger2023auxiliary}. One can then construct and solve the SDP \cref{MeasureSDPData}, obtaining approximations of invariant measures even for continuous-time or stochastic systems. The only difference compared to discrete-time deterministic maps is that the data snapshots $\{(x_i,z_i)\}_{i = 1}^m$  must be such that, for a fixed timestep $\tau>0$, $z_i$ is the state of the system at time $t+\tau$ when the state at time $t$ is $x_i$. The continuous-time Lie derivative is then approximated using the matrix
\begin{equation}
	L_{k\ell} = \frac{K_{k\ell} - \Theta_{k\ell}}{\tau},
\end{equation} 
where $K = QP^\dagger$, the matrices $P$ and $Q$ are as in \cref{EDMDmatrices}, and $\Theta_{k\ell}$ is the rectangular matrix satisfying $p_k(x) = \Theta_{k\ell} \,p_\ell(x)$. In the discrete-time stochastic setting one simply fixes $\tau = 1$.

\section{Examples}\label{s:examples}

We now illustrate the data-driven discovery of invariant measures on a series of examples of increasing difficulty. The first two examples demonstrate our approach on the logistic map (\cref{ss:logistic}) and on a stochastic double-well system (\cref{ss:double-well}). Invariant measures for these examples can be easily found by other methods and compared to our data-driven approximations. In \cref{ss:rossler}, instead, we apply our method to approximate the physical measure of the three-dimensional R\"ossler attractor and to extract unstable periodic orbits embedded in it. We do this by identifying discrete-time periodic orbits for a suitably defined Poincar\'e return map. This last example is purely data-driven, in the sense that no analytical expression for the Poincar\'e return map is available. For all examples, we implemented the SDP \cref{MeasureSDPData} using the MATLAB toolbox YALMIP \cite{lofberg2004yalmip} and solved it using MOSEK~\cite{mosek2015mosek}. Code to reproduce our results is available at \href{https://github.com/jbramburger/data-measures}{https://github.com/jbramburger/data-measures}. 

\subsection{The Logistic Map}\label{ss:logistic}

We begin by discovering invariant measures for the scalar logistic map,
\begin{equation}\label{Logistic}
	x_{n+1} = 2x_n - 1,
\end{equation}
on the interval $X=[-1,1]$.
Our goal is to illustrate the performance of our data-driven approach and highlight its ability to identify the system's physical measure as well as various ergodic measures that have an atomic structure. 

\subsubsection{Physical Measures}

The logistic map \cref{Logistic} is one of few nonlinear discrete dynamical systems for which an exact invariant measure is known \cite{jakobson1981absolutely}. The density of this invariant measure is
\begin{equation}\label{LogisticDensity}
	\rho_*(x) = \frac{1}{\pi\sqrt{1 - x^2}}.
\end{equation}
Here, we seek to `discover' this density from data collected by simulating $m$ iterations of \cref{Logistic} starting from $x_0 = 0.25$. We averaged the ensuing trajectory $\{x_n\}_{n=1}^m$ to obtain an approximate value $\tilde{y}_1$ for the first moment $y_1$ of the physical invariant measure, and we then implemented the SDP \cref{MeasureSDPData} with the objective function 
\begin{equation}
	F(y) = (y_1 - \tilde{y}_1)^2.
\end{equation}
Observe that the value of $\tilde{y}_1$ changes as the amount of data is increased and converges to the exact value $y_1=0$ as $m\to \infty$. 
For numerical stability, we used ChebFun \cite{driscoll2014chebfun} to represent polynomials in the Chebyshev basis $p_k(x) = \{T_0(x),T_1(x),\dots,T_k(x)\}$, rather than the monomial basis. We also took $\ell = 2k$, so $p_\ell(x) = \{T_0(x),T_1(x),\dots,T_{2k}(x)\}$, but nearly identical results are obtained with $\ell > 2k$ because this brings no notable improvement in the data-driven approximation of the Lie derivative of degree-$k$ polynomials \cite[Section~5.3]{bramburger2023auxiliary}.

Optimal solutions to \cref{MeasureSDPData} were used to construct signed polynomial densities $\rho_k$ for the underlying approximate invariant measure as explained in \cref{ss:recovery}. We fixed the reference measure $\pi$ to be the Lebesgue measure on $[-1,1]$ and the results for $k=5$, $10$ and $20$ are compared to the true invariant density $\rho_*$ in \cref{fig:LogisticMeas}. 
It is clear that the signed densities $\rho_k$ approximate $\rho_*$ well, but they also exhibit a highly oscillatory behaviour. Such oscillations are consistent with the fact that, as both the number $m$ of data snapshots and the polynomial degree $k$ become infinite, $\rho_k$ is expected to converge to $\rho_*$ only weakly in $L^1$ (cf. \cref{ss:recovery}), which does not imply pointwise convergence. Nevertheless, the moments of the approximate densities converge to exact values, so our approximations can be used to make meaningful statistical predictions.

\begin{figure}[t] 
	\centering
	\includegraphics[width = \textwidth]{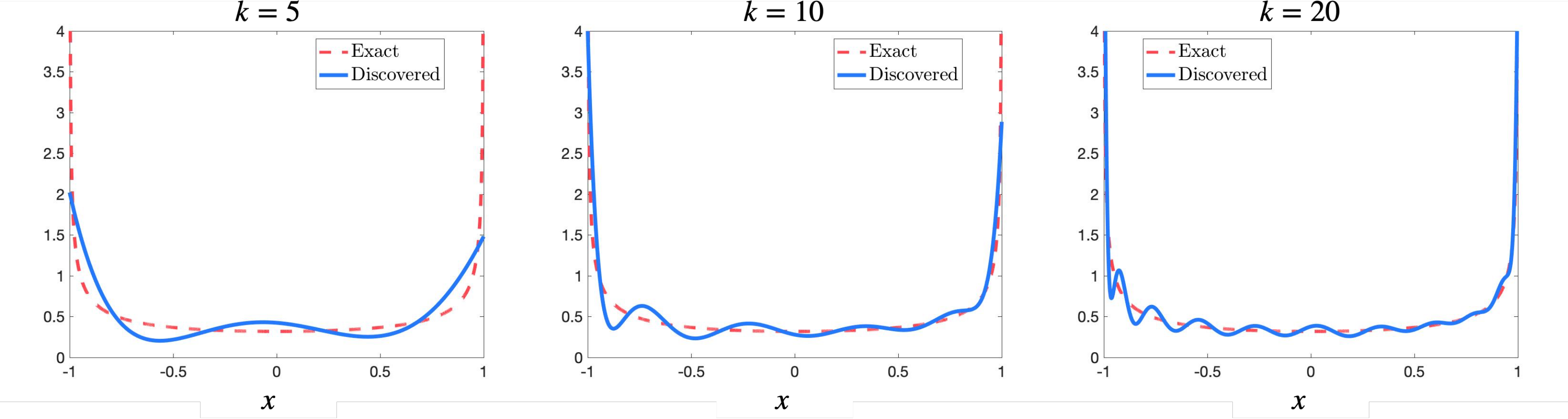}  
	\caption{Discovered densities (blue, solid) for the logistic map \cref{Logistic} using $m = 10^4$ data points compared to the exact density \cref{LogisticDensity} (red, dashed).}
	\label{fig:LogisticMeas}
\end{figure} 

To corroborate this claim, we examined the $L^1$ norm difference between the approximate and exact cumulative distribution functions,
\begin{equation}\label{LogisticDistributions}
	R_k(x) = \int_{-1}^x \rho_k(s) \mathrm{d}s
	\qquad\text{and}\qquad
	R_*(x) = \int_{-1}^x \rho_*(s)\mathrm{d}{s} =  \frac{1}{2} + \frac{1}{\pi}\arcsin(x).
\end{equation}
Indeed, the weak convergence of $\rho_k$ to $\rho_*$ is equivalent to the pointwise convergence of $R_k$ to $R_*$ at every point of continuity in $R_*$ \cite[Theorem~13.23]{klenke2013probability}. Since $R_*$ is continuous for all $x \in [-1,1]$ and $0 \leq R_k(x) \leq 1$ for all $x$ and $k$, the dominated convergence theorem guarantees that $R_k$ converges to $R_*$ in the $L^1$ norm.
The $L^1$ norm difference between $R_k$ and $R_*$ is reported in \cref{tbl:Logistic} for increasing polynomial degrees $k$ and increasing size $m$ of the training data. The table also reports the $L^1$ norm differences obtained by solving the SDP \cref{MeasureSDP}, which uses the exact Lie derivative, and the exact first moment value $\tilde{y}_1 = 0$. For $m \geq 10^4$ the data-driven results closely match those with exact knowledge of the system. The main limitation seems to be approximating the exact invariant measure using semidefinite programming with finite polynomial degree $k$.
This is likely due to the singular behaviour of the optimal density $\rho_*$ at the boundaries of the interval $[-1,1]$, which cannot be approximated well using polynomials of low degree.

\begin{table}[t]
	\centering
	\caption{$L^1$ norm difference between the cumulative distributions \cref{LogisticDistributions} of the discovered and exact densities. Numerical procedures use $m$ consecutive data points under the dynamics of the Logistic map seeded by the initial condition $x_0 = 0.25$ with an objective to fit the approximate first moment $\tilde{y}_1$ extracted from the data. The values $\tilde{y}_1$ are reported in the final column.} Exact values are computed using the exact Lie derivative and first moment $\tilde{y}_1 = 0$ in place of those that are estimated from data.
	\begin{tabular}{ r c cccc|cc}
		\toprule
		& $k=5$ & $k=10$ & $k=15$ & $k=20$ & $k=25$ & $\tilde{y}_1$  \\ [0.5ex]
		\midrule
		$m = 10^2$ & 0.04488 & 0.04401 & 0.04376 & 0.04336 & 0.04325 & -0.04123 \\ 
		$m = 10^3$ & 0.02717 & 0.01602 & 0.01474 & 0.01440 & 0.01389 &  -0.01339 \\
		$m = 10^4$ & 0.02872 & 0.00955 & 0.00531 & 0.00372 & 0.00285 & -0.00102 \\
		$m = 10^5$ & 0.02891 & 0.00959 & 0.00545 & 0.00371 & 0.00278 & -0.00035 \\
		\midrule
		Exact & 0.02902 & 0.00962 & 0.00556 & 0.00378 & 0.00280 & 0.0000    \\
		\bottomrule
	\end{tabular}
	\label{tbl:Logistic}
\end{table}

There are of course other ways to approximate the density of a physical measures from data, such as building a histogram. For demonstration, we take our density used in \Cref{tbl:Logistic} with $m = 10^3$ and $k = 20$ and compare it to the density approximated by the histogram constructed from the same data and organized into 101 bins. The resulting histogram is plotted in \Cref{f:histogram}, where it is compared to the exact density \eqref{LogisticDensity} and to the approximate density obtained with our SDP method and the same data. To evaluate the accuracy of the two approximate densities, we compare their moments against the exact values. Relative errors of these predictions are presented in \Cref{tbl:Histogram} for the first five nonzero moments ($y_2$, $y_4$, $y_6$, $y_8$, and $y_{10}$), and three higher-order moments ($y_{20}$, $y_{30}$, and $y_{40}$). Notice that while the relative errors are comparable for the lower-order moments, the difference in the methods is clearly observed in the approximation of the higher-order moments. In particular, the histogram method provides a poor approximation of the higher-order moments, nearly doubling the relative error of the method of this manuscript. Of course, better accuracy can be achieved by using more data, but our investigations reveal that our SDP-based approach can achieve greater accuracy, particularly for the higher-order moments.

\begin{figure}[t]
	\centering
	\includegraphics[width=0.75\linewidth]{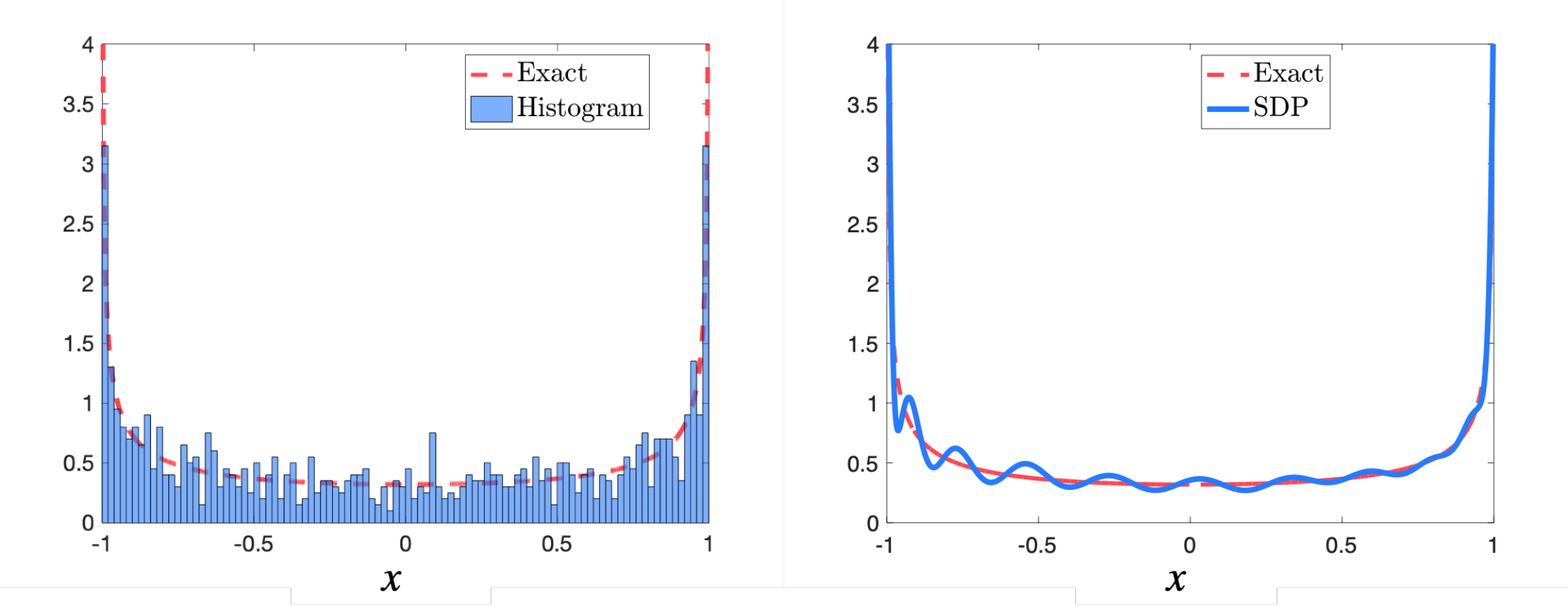}
	\caption{Left: Histogram approximation (blue) of the density of the physical measure \eqref{LogisticDensity} of the logistic map against the exact density (red, dashed). We use $m = 10^3$ datapoints along a single trajectory starting with $x_0 = 0.25$ and organize data into 101 bins. Right: The approximated density using the SDP methods of this manuscript (blue) against the true density (red, dashed) using the same datapoints as the left panel and taking $k = 20$.
	\label{f:histogram}}
\end{figure}	

\begin{table}[t] 
	\centering
	\caption{Relative errors in approximating the moments of \eqref{LogisticDensity} via the approximated densities from the semidefinite programming (SDP) method of \Cref{sec:InvMeasData} and from building a histogram (Histogram) from the data. Both methods use the same $m = 10^3$ datapoints, while the histogram organizes the data into 101 bins covering the domain $[-1,1]$. }
	\begin{tabular}{ r | c ccccccc}
		\toprule
		Moment & $y_2$ & $y_4$ & $y_6$ & $y_8$ & $y_{10}$ & $y_{20}$ & $y_{30}$ & $y_{40}$  \\ [0.5ex]
		\midrule
		SDP & 1.33\% & 2.23\% & 2.81\% & 3.25\% & 2.63\% & 5.04\% & 6.11\% & 7.02\%  \\ 
		Histogram & 1.51\% & 2.62\% & 3.49\% & 4.26\% & 4.99\% & 8.35\% & 11.78\% & 15.46\%   \\
		\bottomrule
	\end{tabular}
	\label{tbl:Histogram}
\end{table}

\subsubsection{Ergodic Measures}

The logistic map \cref{Logistic} has a chaotic attractor that is densely filled by an infinite number of unstable periodic orbits, each of which is associated with an ergodic measure with an atomic structure. As discussed in \cref{sec:Objectives}, moments of such ergodic measures can be approximated by solving the exact SDP \cref{MeasureSDP} or its data-driven approximation \cref{MeasureSDPData} with a linear objective function. Moreover, one can approximate these atomic measures using absolutely continuous ones as explained in \cref{ss:recovery}. In contrast, a data-based histogram cannot provide approximations of these atomic measures.

To illustrate this, we solved the SDP \cref{MeasureSDPData} for polynomial degrees $k = 5$,$10$ and $20$ and the linear cost function $F(y)=y_1$. The data snapshots are as for the discovery of physical measures above. \Cref{fig:LogisticAtomic} shows the densities of the Lebesgue-absolutely-continuous signed measures recovered from the SDP solution. These densities become increasingly localized near $x = -0.5$, which is a fixed point for the logistic map. This suggests that our approximations are converging (in a weak-star sense) to the invariant atomic measure supported on this fixed point. In fact, if the optimal solution $y$ of the SDP \cref{MeasureSDPData} is such that the moment matrix $M(y)$ satisfies a rank deficiency condition \cite[Theorem 5.20]{laurent2009sums}, then one can extract an atomic measure from $y$ \cite[Section~6.7]{laurent2009sums}. We refer interested readers to this work and the references therein for the technical details. For our logistic map example, this extraction procedure\footnote{Implemented in the function \texttt{extractminimizers} from the modified version of YALMIP available at: \href{https://github.com/aeroimperial-optimization/aeroimperial-yalmip}{https://github.com/aeroimperial-optimization/aeroimperial-yalmip}.} returns at atomic measure supported at $x\approx-0.5$ even for $k = 5$, where the approximate density in \cref{fig:LogisticAtomic} is far from being localized.

\begin{figure}[t] 
	\centering
	\includegraphics[width = \textwidth]{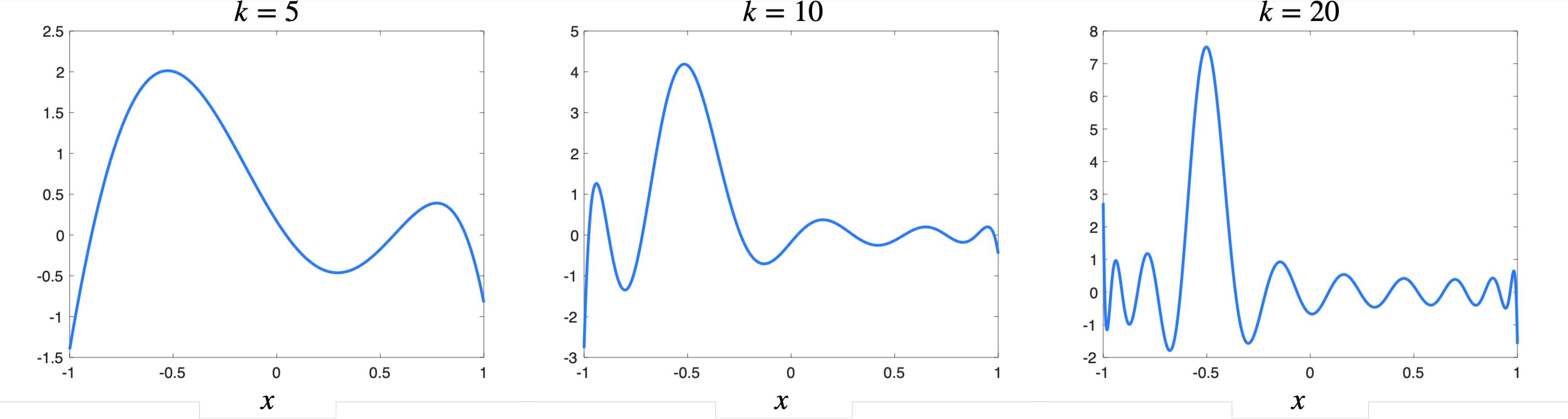}  
	\caption{Polynomial densities of Lebesgue-absolutely-continuous measures that approximates the atomic invariant measure supported on the fixed point $x = -0.5$ of the logistic map \cref{Logistic}. Results are for $m = 10^4$ data snapshots and polynomial degrees $k=5$, $10$ and $20$.}
	\label{fig:LogisticAtomic}
\end{figure}

In a similar way, by varying the linear optimization objective we are able to discover periodic points of the logistic map. For example, fixing $(m,k) = (10^4,5)$ and taking the optimization objective to be $F(y) = y_3$ we are able to extract from the optimal solution of the SDP \cref{MeasureSDPData} an atomic measure supported on the points $x \approx -0.93969$, $0.17365$ and $0.76604$, which approximate a period-3 orbit of the logistic map to four decimal places.  

\begin{remark}\label{rem:ergodic-vs-not} 
While linear objective functions are always optimized by at least one ergodic measure, our approach is not guaranteed to find it. Indeed, suppose two ergodic measures optimize the same objective. Then, solutions of the SDPs \cref{MeasureSDP,MeasureSDPData} may converge to moments of any convex combination of these ergodic measures, which remains optimal but is not ergodic. For instance, with $F(y) = y_5$ in the above example we discover an atomic measure supported on the fixed point $x = -0.5$ and on a period-4 orbit of the logistic map, which is a convex combination of the ergodic measures corresponding to each of these two invariant trajectories. However, in the absence of symmetry or other specific structure in the data, we expect such examples not to be generic.
\end{remark}

\subsection{A Stochastic Double-Well System}\label{ss:double-well}

For our next example we consider the stochastic differential equation
\begin{equation}\label{e:DoubleWell}
	\begin{aligned}
		\dot{x}_1 &= -16(x_1 + x_2)^3 + 2x_1 + 6x_2 + \sigma \xi_1(t),\\
		\dot{x}_2 &= -16(x_1 + x_2)^3 + 6x_1 + 2x_2  + \sigma \xi_2(t),
	\end{aligned}	
\end{equation}
where $\xi_1(t)$ and $\xi_2(t)$ are independent Gaussian white noise processes and $\sigma$ determines the noise intensity. The unique invariant measure associated to this equation is absolutely continuous with respect to the Lebesgue measure on $\mathbb{R}^2$ and has the density
\begin{equation}\label{e:double-well-exact-sol}
	\rho(x_1,x_2) = C \exp\left\{-\frac{[  4(x_1+x_2)^2 - 1 ]^2 + 4(x_1-x_2)^2}{2\sigma^2} \right\},
\end{equation}
where the constant $C$ is such that $\int_{\mathbb{R}^2}\rho(x_1,x_2)\mathrm{d}x_1\mathrm{d}x_2 = 1$. This density can be found by solving the stationary Fokker-Plank equation associated to \cref{e:DoubleWell}. As illustrated in the left-most panel of \cref{f:double-well-comparison}, $\rho$ exhibits two peaks at $(x_1,x_2)=(\frac14,\frac14)$ and $(x_1,x_2)=(-\frac14,-\frac14)$, which are the stable fixed points of the deterministic ($\sigma = 0$) counterpart of \cref{e:DoubleWell}. The peaks become increasingly localized as $\sigma$ decreases. In this example, we use our data-driven discovery approach to approximate this two-peak invariant density for $\sigma=0.75$. This value gives two clear peaks, which are however sufficiently "spread out" for $\rho$ to be approximated well by polynomials of relatively small degree. Smaller $\sigma$ could be considered at the cost of working with larger polynomial degrees (hence, larger computational cost) than those reported below.

\begin{figure}
	\centering
	\includegraphics[width=\linewidth,trim=1.5cm 0.75cm 0cm 0.75cm,clip]{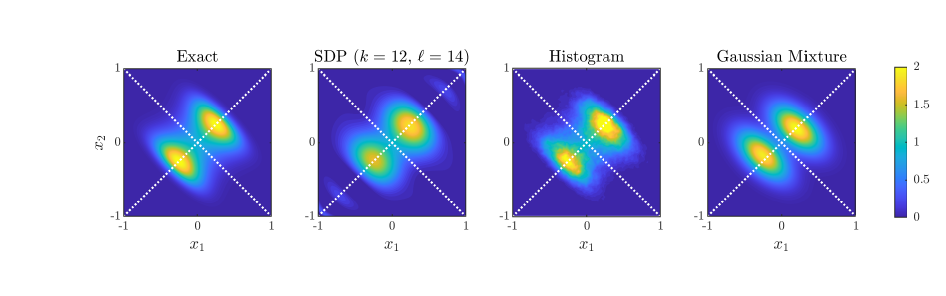}
	\caption{Left-most panel: Invariant density $\rho$ for the stochastic system in \cref{e:DoubleWell} for with $\sigma=0.75$. The formula for $\rho$ is given by \cref{e:double-well-exact-sol}. Other panels: Approximations to $\rho$ obtained from the dataset $\mathcal{D}_1$ with (in order) with our SDP approach, with a data-based histogram, and with a two-component Gaussian mixture fit to the data. In each panel, dotted white lines mark the reflection symmetry axes ($x_2=x_1$ and $x_2=-x_1$) for the exact invariant density.
	\label{f:double-well-comparison}}
\end{figure}

We consider two sets of data snapshots $(x(t),x(t+\tau))$ in $X=\mathbb{R}^2$, both collected by approximating solutions of \cref{e:DoubleWell} with the Euler--Maruyama scheme and timestep $\tau = 10^{-4}$. The first dataset, referred to as $\mathcal{D}_{500}$ in what follows, was obtained from $500$ realizations of the stochastic process, integrated for $10^4$ timesteps starting from $500$ initial conditions sampled from the uniform distribution on the square $[-\frac12,\frac12]$. The second dataset, denoted by $\mathcal{D}_1$, was instead obtained from a single realization with $5\times 10^6$ timesteps, again started from a random initial condition on the square $[-\frac12,\frac12]$. Datapoints $(x(t),x(t+\tau))$ not in the unit box $[-1,1]^2$ were dropped from both datasets to ensure a numerically stable evaluation of Chebyshev polynomials (see the next paragraph). As a result, the dataset $\mathcal{D}_{500}$ contains $m=4\,999\,759$ snapshots from $500$ process realizations, while the dataset $\mathcal{D}_{1}$ has $m=4\,998\,925$ snapshots from a single, but longer, process realization. The datasets are included in the code repository accompanying this article.

For each dataset, we solved the SDP \cref{MeasureSDPData} for increasing polynomial degree $k$ and polynomial EDMD dictionaries of degrees $\ell = k$, $k+2$ and $k+4$. As before, for numerical stability we represent polynomials using the Chebyshev basis, rather than the monomial basis. We also found that the quality of approximate Lie derivatives could be considerably improved if we computed the matrix $L_{k\ell}$ with an iterative procedure similar to the SINDy method \cite{SINDy}. Briefly, for each row of the matrix, we find the entry with largest magnitude, set to zero entries with magnitude smaller than 0.1\% of this value, and re-optimize the others. The process can be repeated until the matrix $L_{k\ell}$ stabilizes. Finally, the objective function for the SDP \cref{MeasureSDPData} was set to $F(y)= -2 - y_{20} - y_{02}$, which corresponds to looking for the invariant measure that maximizes the stationary expectation of the observable function $x_1^2 + x_2^2$. Results for different objective functions (omitted for brevity) are similar, but are not identical even though \cref{e:DoubleWell} admits a unique invariant measure. This is expected because the constraints of \cref{MeasureSDPData} do not characterize invariant measures exactly.

\begin{figure}
	\centering
	\includegraphics[width=\linewidth,trim=2cm 0.5cm 1cm 0cm,clip]{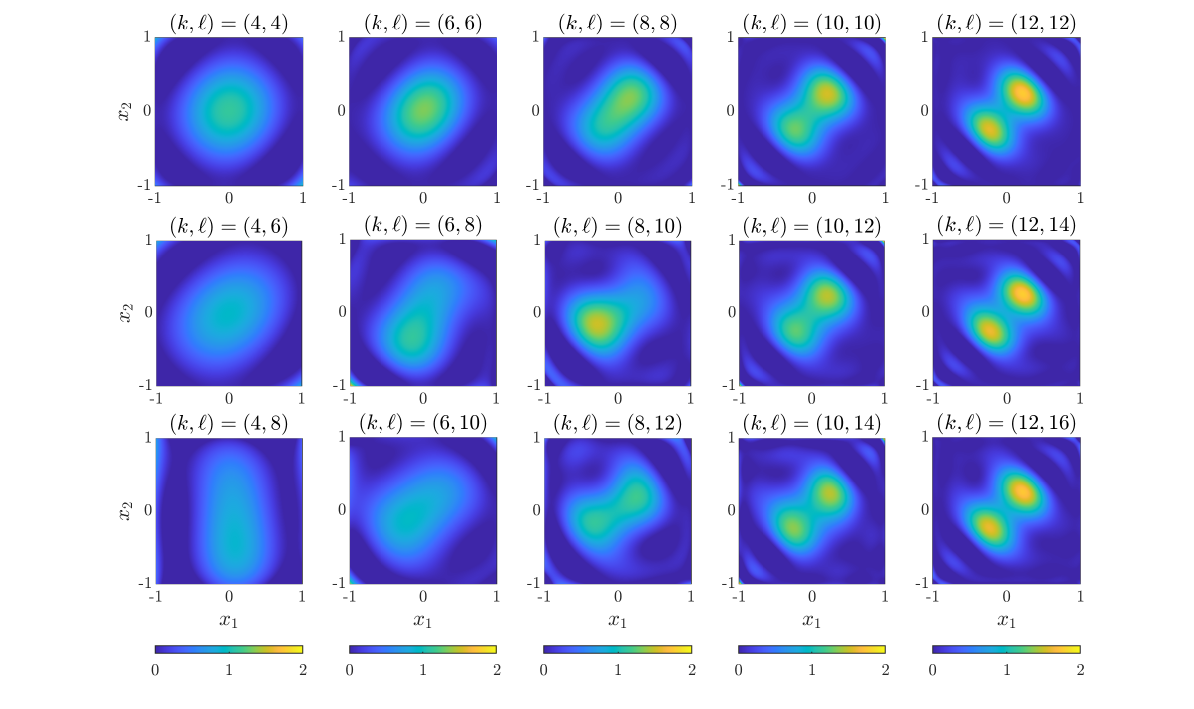}
	\caption{Data-driven approximations of the invariant density for the stochastic double-well system \cref{e:DoubleWell}, recovered from the solution of the SDP \cref{MeasureSDPData}. Different panels correspond to different choices of the polynomial degrees $k$ and $\ell$, which are indicated in each case. All plots are for the dataset $\mathcal{D}_{500}$.
	\label{f:double-well-multi-traj}}
\end{figure}

\begin{figure}
	\centering
	\includegraphics[width=\linewidth,trim=2cm 0.35cm 1cm 0cm,clip]{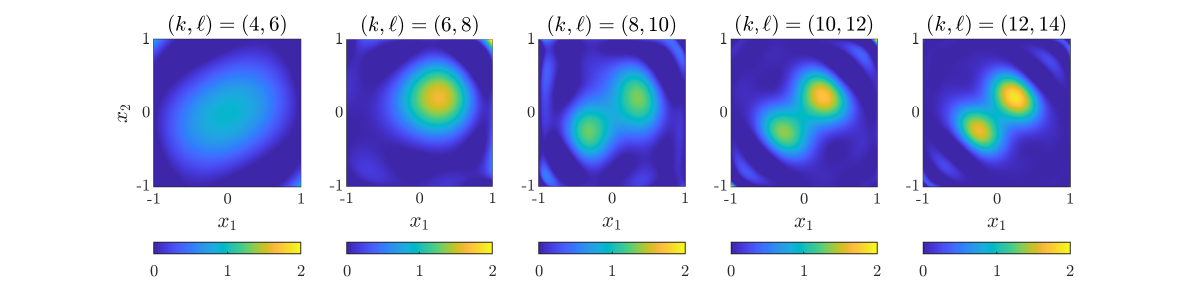}
	\caption{Data-driven approximations of the invariant density for the stochastic double-well system \cref{e:DoubleWell}, recovered from the solution of the SDP \cref{MeasureSDPData}.
	Different panels correspond to different choices of the polynomial degrees $k$ and $\ell$, which are indicated in each case. All plots are for the dataset $\mathcal{D}_{1}$.
	\label{f:double-well-single-traj}}
\end{figure}

\begin{table}
	\centering
	\caption{Stationary expectations of bivariate Chebyshev polynomials $T_\alpha(x_1)T_\beta(x_2)$ for the stochastic system \eqref{e:DoubleWell}, predicted using (i) the approximate densities found with the SDP method of \Cref{sec:InvMeasData}, (ii) a histogram, and (iii) a two-component Gaussian mixture model. Results are for the dataset $\mathcal{D}_1$ and all values $\alpha,\beta$ such that $\alpha+\beta \in \{2,4,6\}$. Errors are with respect to the exact expectation values $\int_{\mathbb{R}^2} T_\alpha(x_1)T_\beta(x_2) \rho(x_1,x_2)\mathrm{d}x_1\mathrm{d}x_2$, calculated using the exact invariant density $\rho$ from \cref{e:double-well-exact-sol} and reported in the first column.
	\label{table:moments-2d}}
	\small
	\begin{tabular}{c c r c rr c rr c rr}
		\toprule
		&& \multicolumn{1}{c}{Exact} && \multicolumn{2}{c}{SDP} 
		&& \multicolumn{2}{c}{Histogram} && \multicolumn{2}{c}{Gaussian Mixture}\\
		\cline{5-6}\cline{8-9}\cline{11-12}
		\\[-2ex]
		&& \multicolumn{1}{c}{Value} 
		&& \multicolumn{1}{c}{Value} & \multicolumn{1}{c}{Error}
		&& \multicolumn{1}{c}{Value} & \multicolumn{1}{c}{Error}
		&& \multicolumn{1}{c}{Value} & \multicolumn{1}{c}{Error}
		\\\midrule
$T_{2}(x_1) T_{0}(x_2)$ && -0.825 && -0.816 &    1.1\% && -0.817 &    1.0\% && -0.796 &    3.6\% \\
$T_{1}(x_1) T_{1}(x_2)$ &&  0.017 &&  0.019 &    9.5\% &&  0.019 &   10.4\% &&  0.015 &   14.7\% \\
$T_{0}(x_1) T_{2}(x_2)$ && -0.825 && -0.840 &    1.8\% && -0.841 &    1.9\% && -0.842 &    2.0\% \\
$T_{4}(x_1) T_{0}(x_2)$ &&  0.453 &&  0.431 &    4.9\% &&  0.432 &    4.6\% &&  0.390 &   13.8\% \\
$T_{3}(x_1) T_{1}(x_2)$ && -0.049 && -0.054 &   10.0\% && -0.054 &   10.3\% && -0.043 &   11.6\% \\
$T_{2}(x_1) T_{2}(x_2)$ &&  0.668 &&  0.673 &    0.8\% &&  0.674 &    1.0\% &&  0.658 &    1.4\% \\
$T_{1}(x_1) T_{3}(x_2)$ && -0.049 && -0.051 &    5.3\% && -0.052 &    5.7\% && -0.041 &   16.8\% \\
$T_{0}(x_1) T_{4}(x_2)$ &&  0.453 &&  0.488 &    7.7\% &&  0.490 &    8.1\% &&  0.500 &   10.4\% \\
$T_{6}(x_1) T_{0}(x_2)$ && -0.141 && -0.118 &   16.3\% && -0.118 &   16.1\% && -0.087 &   38.1\% \\
$T_{5}(x_1) T_{1}(x_2)$ &&  0.066 &&  0.073 &    9.8\% &&  0.072 &    9.2\% &&  0.058 &   12.3\% \\
$T_{4}(x_1) T_{2}(x_2)$ && -0.337 && -0.326 &    3.0\% && -0.328 &    2.5\% && -0.296 &   12.0\% \\
$T_{3}(x_1) T_{3}(x_2)$ &&  0.137 &&  0.146 &    6.4\% &&  0.146 &    6.2\% &&  0.118 &   13.8\% \\
$T_{2}(x_1) T_{4}(x_2)$ && -0.337 && -0.362 &    7.5\% && -0.364 &    8.2\% && -0.364 &    8.3\% \\
$T_{1}(x_1) T_{5}(x_2)$ &&  0.066 &&  0.068 &    2.2\% &&  0.068 &    2.5\% &&  0.054 &   18.5\% \\
$T_{0}(x_1) T_{6}(x_2)$ && -0.141 && -0.172 &   22.0\% && -0.174 &   23.4\% && -0.200 &   42.0\% \\
		\bottomrule
	\end{tabular}
\end{table}

\Cref{f:double-well-multi-traj} shows approximate signed invariant densities for \cref{e:DoubleWell} obtained with the dataset $\mathcal{D}_{500}$. Results for the dataset $\mathcal{D}_{1}$ are similar, so we report only a subset of them in \cref{f:double-well-single-traj}. In both cases, the approximate signed invariant densities were recovered from the optimal SDP solutions as explained in \cref{ss:recovery}, taking the reference measure $\pi$ to be the Lebesgue measure on $[-1,1]^2$. Just like the true invariant density $\rho$ in \cref{e:double-well-exact-sol}, the signed densities localize near the points $(\frac14,\frac14)$ and $(-\frac14,-\frac14)$ as the polynomial degree $k$ is raised. For sufficiently large values of the parameters $k$ and $\ell$, moreover, they approximately satisfy the reflection symmetry of the exact invariant density $\rho$ with respect to the lines $x_2=x_1$ and $x_2=-x_1$. The lack of exact symmetry in the figures can be attributed to a lack of symmetry in our datasets because, as illustrated by \cref{f:double-well-comparison}, the same issue arises also if we approximate $\rho$ by fitting a probability histogram or a two-component Gaussian mixture model to the data. (This can be done with the MATLAB functions \texttt{histogram2} and \texttt{fitgmdist}, respectively, which we used with default parameters.) Incidentally, as evidenced by \cref{table:moments-2d}, the invariant densities produced by the SDP approach and by fitting a histogram predict the stationary expectation of bivariate Chebyshev polynomials $T_\alpha(x_1)T_\beta(x_2)$ with similar accuracy. In contrast, the two-component Gaussian mixture model has much larger prediction errors.

To conclude this example, we point out that the results for $\ell=k+2$ and $\ell=k+4$ in \cref{f:double-well-multi-traj} become almost indistinguishable as $k$ increases. This is to be expected because the right-hand side of \cref{e:DoubleWell} is a cubic polynomial, so \cref{e:sde-lie-derivative} and the convergence results in \cite{bramburger2023auxiliary} guarantee that the choice $\ell = k+2$ is large enough to recover exactly the Lie derivative of degree-$k$ polynomials as the number of data snapshots tends to infinity.

\subsection{The R\"ossler System}\label{ss:rossler}

Having demonstrated our data-driven approach to discovering invariant measures on two relatively simple examples, we now showcase its usefulness in providing insight into the statistics of a chaotic dynamical system. Precisely, we consider the chaotic R\"ossler system
\begin{equation}\label{Rossler}
	\begin{split}
		\dot x_1 &= -x_2 - x_3, \\
		\dot x_2 &= x_1 + 0.1x_2, \\
		\dot x_3 &= 0.1 + x_3(x_1 - 18),
	\end{split}
\end{equation}
and show that our data-driven approach can produce good approximations of the statistics on the attractor, as well as enable the discovery of a large number of unstable periodic orbits. In both cases, the dataset consists of a numerically integrated trajectory of \cref{Rossler} starting from $(x_1,x_2,x_3)=(0,-20,0)$ at $t = 0$, with data collected every $\tau = 0.005$ time units up to $t = 1000$. We look for an invariant measure in the set $X = [-30,30]\times [-30,30]\times [0,60]$, which contains the attractor. For computations, variables are scaled so that $X$ is the unit cube $[-1,1]^3$, but results are plotted in the original coordinates of \cref{Rossler}.

\subsubsection{Physical Measures and Attractor Statistics}
First, we approximate the physical measure on the chaotic attractor by solving \cref{MeasureSDPData} with polynomial degrees $(k,\ell) = (14,15)$ and optimization objective 
\begin{equation}\label{RosslerObj}
	F(y) = \frac{(y_{100} - \tilde y_{100})^2}{\tilde y_{100}^2} + \frac{(y_{010} - \tilde y_{010})^2}{\tilde y_{010}^2} + \frac{(y_{001} - \tilde y_{001})^2}{\tilde y_{001}^2}.
\end{equation} 
The values $\tilde y_{100}$, $\tilde y_{010}$ and $\tilde y_{001}$ are the approximate moments of $x_1$, $x_2$ and $x_3$, respectively, which we obtained from the training data via \cref{ApproxMoment}. We used the method of \cref{ss:recovery} with the reference measure $\pi$ fixed to the Lebesgue measure on $X$ in order to recover an absolutely-continuous measure with signed polynomial density of degree $k=14$ that approximates the attractor's physical measure.

\begin{figure}[t] 
	\centering
	\includegraphics[width = \textwidth]{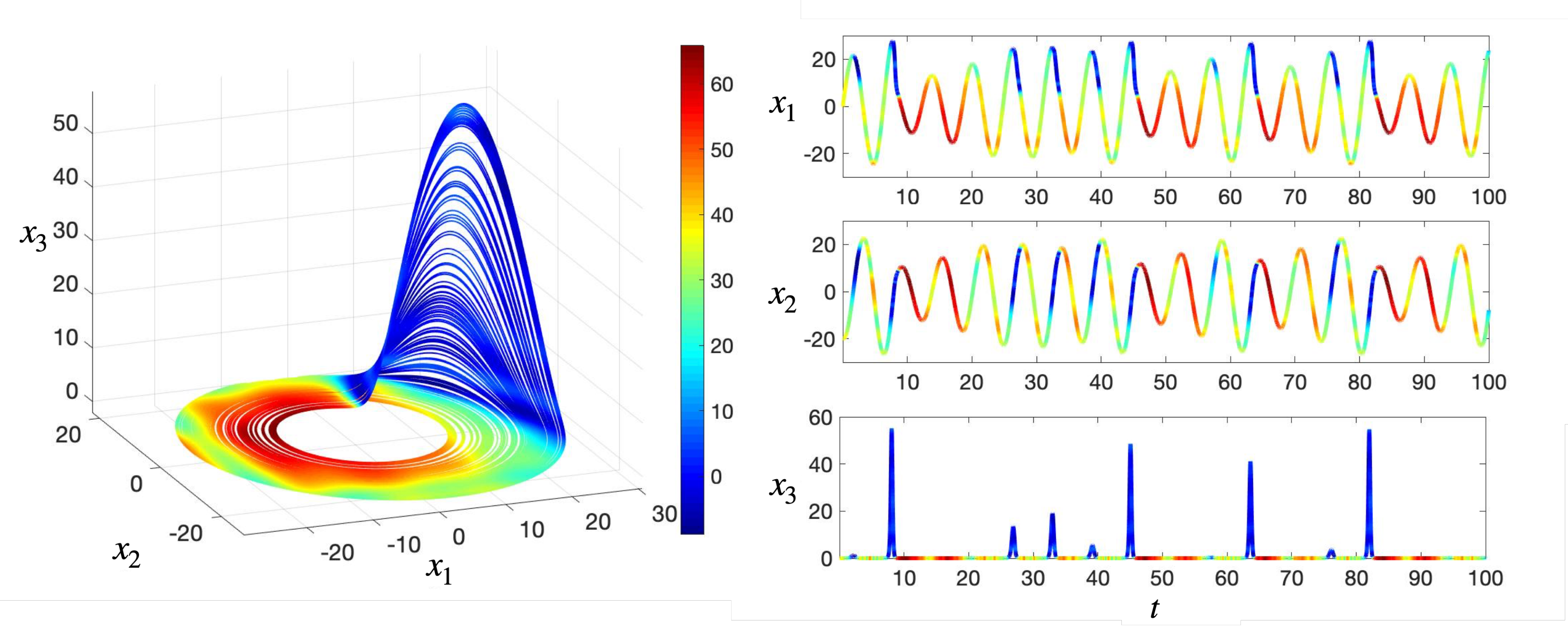}  
	\caption{Simulated data of the R\"ossler system \cref{Rossler} coloured according to the value of the discovered approximate density using $(k,\ell) = (14,15)$ and the objective \cref{RosslerObj}. Data on the chaotic attractor in $(x_1,x_2,x_3)$-space is on the left, while timeseries plots for each variable are on the right.}
	\label{fig:RosslerRainbow}
	\end{figure} 

The left panel in \cref{fig:RosslerRainbow} shows the training data points on the system's attractor, coloured according to the value of the discovered signed density. The large upward swings in $x_3$ along the attractor are identified as rare events, which is confirmed by the three timeseries plots on the right of \cref{fig:RosslerRainbow}. Moreover, the highest density component is in the region $x_1 \leq 0$, through which all trajectories seem to be funnelled as they circle around the attractor. Finally, all moments of the discovered density up to degree 2 agree with time averages integrated up to $t = 10^6$ to within a $13\%$ error, showing that we can make accurate statistical predictions even for moments not included in the objective function.

\subsubsection{Discovery of Unstable Periodic Orbits}
Next, we turn to the discovery of UPOs for the R\"ossler system through the analysis of a Poincar\'e return map. 
This is a truly data-driven example because the Poincar\'e return map for \cref{Rossler} does not have a closed-form analytical expression. Yet, analysing it is convenient because its UPOs have an atomic structure and, therefore, can more easily be detected within the polynomial-based convex optimization framework of \cref{sec:InvMeas,sec:InvMeasData}.

Numerical evidence suggests that the R\"ossler system has the advantageous property that points on the attractor with $x_1 = 0$ also have $x_3 = 0$. Therefore, to define the Poincar\'e return map we choose the Poincar\'e section to be the plane $x_1 = 0$ with the crossing condition $\dot x_1 > 0$ and only track the value of $x_2$ at each crossing. We then simulate \cref{Rossler} from time $t = 0$ to time $t = 10^6$ starting from $(x_1,x_2,x_3)=(0,-20,0)$ and measure the values of $x_2$ every time the simulated trajectory crosses the Poincar\'e section. This produces $811$ values $x_{2,i}$, $i=1,\ldots,811$, which we organize in $m=810$ data snapshots $(x_{2,i},x_{2,i+1})$ corresponding to measurements of the Poincar\'e return map. These datapoints are plotted in the top-left panel of \cref{fig:Rossler_Pmap}, which reveals that the Poincar\'e map is unimodal. However, the exact mapping is unknown and can only be estimated from data.   

\begin{figure}[t] 
\centering
\includegraphics[width = \textwidth]{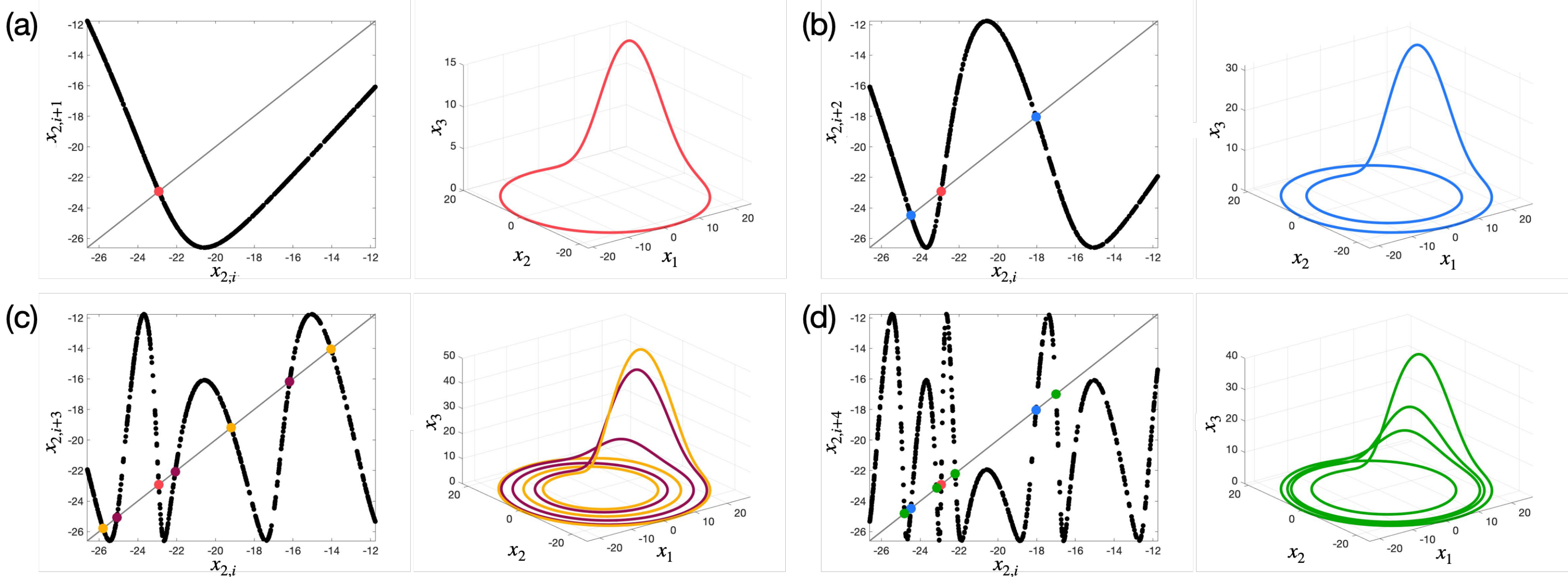}  
\caption{Poincar\'e map data and UPOs for the R\"ossler system \cref{Rossler}. Each panel has two subfigures. \emph{Left subfigure:} Data points on the Poincar\'e section, plotted as $(x_{2,i},x_{2,i+n})$ to represent measurements of the $n$th iterate of the Poincar\'e map with (a) $n = 1$, (b) $n = 2$, (c) $n = 3$, and (d) $n = 4$. Intersections with the diagonal $(x_2,x_2)$ correspond to period-$n$ fixed points of the Poincar\'e map, which are coloured according to their minimal period. \emph{Right subfigure:} UPOs embedded in the continuous-time R\"ossler attractor corresponding to the period-$n$ UPOs of the Poincar\'e map.}
\label{fig:Rossler_Pmap}
\end{figure} 
 
The UPOs of the R\"ossler attractor are in one-to-one correspondence with the UPOs of the Poincar\'e map. We thus attempt to extract the latter by discovering invariant atomic measures from the data collected in the Poincar\'e section. For accuracy, we solve the SDP \cref{MeasureSDPData} with polynomial degrees $(k,\ell) = (20,80)$, meaning that we are approximating the Lie derivative of polynomials in $x_2$ up to degree $k=20$ using polynomials up to degree $\ell=80$. Loosely speaking, this corresponds to assuming the Poincar\'e map is well approximated by a quartic polynomial map. As before, to improve numerical stability we represent polynomials using the Chebyshev basis, rather than the monomial basis. Then, to extract the periodic points from the Poincar\'e map we run the optimization procedure \cref{MeasureSDPData} with 1000 randomized linear objective functions. We are able to extract periodic points of periods 1 through 19, with the low-period points plotted in \cref{fig:Rossler_Pmap}. To confirm the accuracy of our method, the figure also shows these points plotted against the forward iterates of the Poincar\'e map. Observe that, as expected for periodic points, they lie precisely along the $(x_2,x_2)$-diagonal. Importantly, by counting the intersections of the Poincar\'e map data with the diagonal, we conclude we have extracted \emph{all} UPOs of the Poincar\'e map with period up to at least 4.   

\begin{figure}[t] 
\centering
\includegraphics[width = \textwidth]{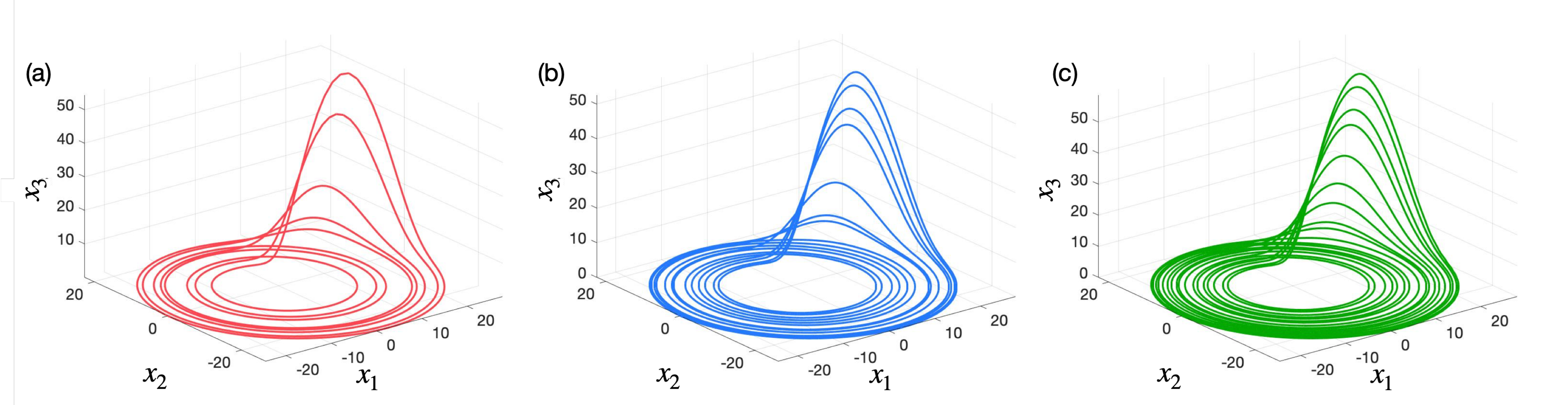}  
\caption{A sample of long period UPOs of the R\'ossler attractor. Plotted UPOs intersect the Poincar\'e section (a) 8 times, (b) 13 times, and (c) 17 times.}
\label{fig:Rossler_Long}
\end{figure} 

The periodic points in the Poincar\'e section can be further used as initial conditions for root-finding methods that compute the corresponding UPOs for the continuous-time R\"ossler system (see \cite[Section~2.3]{bramburger2021deep} for details).
\Cref{fig:Rossler_Pmap} shows all UPOs that intersect the Poincar\'e section up to four times. \Cref{fig:Rossler_Long} provides a small sample of longer UPOs that we are able to extract via the same process. 

A valid alternative to our data-driven discovery of atomic invariant measures for the Poincar\'e map would be to perform model identification on the Poincar\'e section data \cite{bramburger2021data,bramburger2020poincare,bramburger2021deep} and then numerically extract the periodic points from the discovered symbolic mapping. This can perform well for low-period UPOs, but small errors in the model discovery process accumulate through iterations of the discovered map, making it difficult (if not impossible) to identify UPOs with long period \cite{bramburger2021data}. The advantage of our data-driven approach is that we do not rely on iterates of an approximate Poincar\'e map to identify its periodic points, so model discovery errors do not accumulate in the same way.

\section{Conclusions}\label{s:conclusion}

In this article, we have described a convex optimization approach to approximate invariant measures for dynamical systems entirely from data. More precisely, we have shown that the model-based convex optimization framework for the approximation of invariant measures from~\cite{Korda2021invariant} can be implemented in an entirely data-driven manner by replacing Lie derivatives with data-driven approximations obtained with EDMD. Our approach can be applied to a broad range of dynamical systems, including those governed by deterministic or stochastic differential equations and maps. Moreover, and crucially from a practical point of view, the implementation of our method is the same irrespective of whether the underlying dynamics are deterministic or stochastic.

Compared to alternative data-driven approaches, such as Ulam's method or the EDMD-based strategy in \cite{eDMDconv2}, our method has the advantage of enabling one to target a particular invariant measure through the choice of a suitable optimization objective. We have demonstrated through a series of examples that one can successfully approximate the physical measure of a chaotic attractor, as well as ergodic invariant measures associated to different periodic orbits embedded within it. In particular, the success in discovering UPOs for the R\"ossler system suggests that our computational approach can be extremely useful whenever UPOs are of interest. One example of practical relevance is satellite mission design, since orbits that shadow UPOs require minimal use of propellant. Orbital dynamics models such as the reduced three-body system, which has a four-dimensional state space but a Hamiltonian structure that allows for a two-dimensional Poincar\'e section, are only marginally more complicated than the R\"ossler system studied in \cref{ss:rossler}. We believe they can be studied successfully with our method, although this remains to be confirmed.

Finally, although we have demonstrated that our approach works well in practice, questions about its theoretical convergence remain. Known convergence results for EDMD approximations of Lie derivatives \cite{bramburger2023auxiliary} provide only a heuristic justification for the good performance of our approach. Rigorous statements require proving that optimizers of the SDP \cref{MeasureSDPData} converge to optimizers of \cref{MeasureLP} as the number of data snapshots $m$ and the polynomial degrees $k$ and $\ell$ of the EDMD dictionaries increases. We have done this in \cref{sec:Convergence} when the data is governed by polynomial maps satisfying suitable technical conditions. For general nonlinear maps, it should still be relatively straightforward to study the limit of infinite data ($m\to\infty$) for fixed $k$ and $\ell$. It seems however much harder to take $\ell\to\infty$ for fixed $k$, as this requires studying the fine convergence properties of EDMD-based  polynomial approximations to non-polynomial Lie derivatives. The main technical challenges involved in this analysis are absent from our study of polynomial maps in \cref{sec:Convergence} because in this case the infinite-data limit of the matrix $L_{k\ell}$ coincides with the matrix $A_{k\ell}$ in \cref{MeasureSDPExtended} for all $\ell \geq dk$, so our data-driven approach recovers (a straightforward generalization of) the model-based strategy from~\cite{Korda2021invariant} for a finite, albeit unknown, value $\ell$. For general dynamics governed by non-polynomial maps, however, one is forced to take $\ell$ to infinity. Identifying which problem \cref{MeasureSDPData} tends to in this limit and studying how this limiting problem behaves for increasing $k$ remain two important open questions.

\section*{Acknowledgments}
JB was supported by an NSERC Discovery Grant. GF is grateful to Mayur Lakshmi for pointing out the relevance of the methods in this paper for orbital mechanics.

\bibliographystyle{abbrv}
\bibliography{./references.bib}

\end{document}